\documentclass[11pt,a4paper]{article}

\usepackage{bbm,
  a4wide,
  amssymb,
  amsmath,
  amsthm,
}

\usepackage{pgf}
\usepackage{tikz}
\usetikzlibrary{arrows,automata}

\usepackage[colorlinks=true]{hyperref}
\usepackage{pdfsync}

\usepackage[textwidth=0.89in,textsize=scriptsize]{todonotes}

\newcommand{\Prec}{{\sf Prec}}
\newcommand{\pfl}{\xi}

\usepackage[ruled,vlined]{algorithm2e}
\newcommand{\NN}{\mathbb{N}}
\newcommand{\QQ}{\mathbb{Q}}
\newcommand{\ZZ}{\mathbb{Z}}
\newcommand{\RR}{\mathbb{R}}

\newcommand{\leqlex}{\leq_{\rm lex}}
\newcommand{\ltlex}{<_{\rm lex}}
\newcommand{\occ}{{\sf occ}}
\newcommand{\addtag}{\refstepcounter{equation}\tag{\theequation}}

\newcommand{\uh}[1]{\!\upharpoonright_{#1}}
\renewcommand{\dh}[1]{\!\downharpoonright_{#1}}

\theoremstyle{definition}
\newtheorem{theorem}{Theorem}
\newtheorem{lemma}[theorem]{Lemma}
\newtheorem{corollary}[theorem]{Corollary}
\newtheorem{proposition}[theorem]{Proposition}
\newtheorem{question}[theorem]{Question}
\newtheorem{definition}[theorem]{Definition}

\newtheorem{observation}[theorem]{Observation}

\newtheorem{example}[theorem]{Example}
\numberwithin{theorem}{section}

\begin{document}

%\begin{frontmatter}

\title{Normality in non-integer bases and\\polynomial time randomness}
\author{Javier Almarza \and Santiago Figueira}

%\author[DC]{Javier Ignacio Almarza}
%\author[DC,CONICET]  {Santiago Figueira}
%
%\address[DC]{Departamento de Computaci\'on, FCEyN\\Universidad de Buenos Aires, Argentina}
%
%\address[CONICET]{CONICET, Argentina}

\maketitle

\begin{abstract}
It is known that if $x\in[0,1]$ is polynomial time random (i.e.\ no polynomial time computable martingale succeeds on the binary fractional expansion of $x$) then $x$ is normal in any integer base greater than one. We show that if $x$ is polynomial time random  and $\beta>1$ is Pisot, then $x$ is ``normal in base $\beta$'', in the sense that the sequence $(x\beta^n)_{n\in\NN}$ is uniformly distributed modulo one.
We work with the notion of {\em $P$-martingale}, a generalization of martingales to non-uniform distributions, and show that a sequence over a finite alphabet is distributed according to an irreducible, invariant Markov measure~$P$ if an only if no $P$-martingale whose betting factors are computed by a deterministic finite automaton succeeds on it. This is a generalization of Schnorr and Stimm's characterization of normal sequences in integer bases. Our results use tools and techniques from symbolic dynamics, together with automata theory and algorithmic randomness.
\end{abstract}

%\end{frontmatter}

\section{Introduction}

%
%We extend results on randomness and normality to non-integer bases. To this end, a brief introduction to the area of symbolic dynamics, which underlies non-integer representations, will be needed. Sections~\ref{sec:betaexp},~\ref{sec:symdyn} and~\ref{sec:pisot} deal with this introduction, and all results mentioned there can be found in either~\cite{Bertrand86} or~\cite{Kitchens98}. The main contribution of that literature is to establish the existence of a well-behaved measure on the space of expansions of certain non-integer bases. In section~\ref{sec:main}, we extend the definition of martingales to make it suitable under non-uniform measures, and we prove a generalization of Schnorr and Stimm's characterization of normal sequences in terms of our generalized martingales.
%
%----------------

A weak notion of randomness for  sequences over a finite alphabet $\Sigma=\{0,\dots,b-1\}$ ($b\in\NN$) is {\em normality},  introduced by Borel in 1909. Normality may be regarded as a ``law of large numbers" for blocks of events, in the sense that the average occurrences of a block $\sigma\in\Sigma^*$ of length $n$ converges to $|\Sigma|^{-n}$.
A real number $x$ is called \textit{normal in base b} ($b\in\NN$) if its expansion in base $b$ is normal.  While almost all numbers are normal to all bases it is not too difficult to see that this notion is not base invariant. In fact for any multiplicatively independent bases $b$ and $b'$ the set of numbers normal to $b$ but not normal to $b'$ has full Hausdorff dimension~\cite{Polli81}. We say a number $x$ is \textit{absolutely normal} if it is normal in all integer bases greater than one. It is not difficult to see that $x$ is normal in base $b$ if and only if the sequence $(xb^n)_{n\in\NN}$ is u.d.\ modulo one, and then $x$ is absolutely normal if and only if $(xb^n)_{n\in\NN}$ is uniformly distributed (u.d.) modulo one for all integer $b>1$.

Polynomial time randomness is another weak notion of randomness. We say that $x$ is \textit{polynomial time random in base $b$} if no martingale (a formalization of {\em betting strategy}) on the alphabet $\{0,\dots,b-1\}$  which is computable in polynomial time succeeds on the expansion of $x$ in base $b$. A result of Schnorr~\cite{S71} states that if $x$ is polynomial time random in base $b$ then $x$ is normal in base $b$.
It was recently shown~\cite{Fi2013} that polynomial time randomness is base invariant, so that being polynomial time random in a single base implies being normal for all bases, i.e.\ being absolutely normal. The converse is not true, since there are absolutely normal numbers which are computable in polynomial time~\cite{Fi2013,BHS13,LM12}, and these cannot be polynomial time random. The following question was left open in~\cite{Fi2013}:

\begin{question}
Suppose that $x$ is polynomial time random. Is the sequence $(x\beta^n)_{n\in\NN}$  u.d.\ modulo one for all rational $\beta>1$?
\end{question}
The distribution of $(x\beta^n)_{n\in\NN}$ modulo one for rational $\beta$ seems, however, fairly intractable. It is unknown, for instance, if $((3/2)^n)_{n\in\NN}$ is u.d.\ modulo one. Our first main result is that there is a class of algebraic reals for which the question may be readily handled:

\begin{theorem}\label{thm:main1} If $x$ is polynomial time random then the sequence $(x\beta^n)_{n\in\NN}$ is u.d.\ modulo one for all Pisot $\beta>1$.
\end{theorem}

%To prove Theorem~\ref{thm:main1} we generalize a result by Schnorr and Stimm~\cite{Sch72}, which establishes that
%
%\begin{theorem}\label{thm:ss}\label{teoschnorr}A real $x$ is normal in base $b$ if and only if no martingale on the alphabet of $b$ digits whose betting factors are computed by a deterministic finite automaton (DFA) succeeds on the expansion of $x$ in base $b$.
%\end{theorem}

%Like other notions of algorithmic randomness, normality admits a characterization in terms of betting functions, which are called \textit{martingales}.

Observe that any non-integer Pisot $\beta$ is irrational, and as a consequence of a result of Brown, Moran and Pearce \cite[Theorem 2]{Brown.Moran.etal:86}, there are uncountably many reals which are absolutely normal but $(x\beta^n)_{n\in\NN}$ is not u.d.\ modulo one.

The formulation of normality to integer bases $\beta$ in terms of modulo one uniform distribution allows us to understand normality as equivalent to what ergodic theory calls \textit{genericity}, an equivalence which boils down to two facts: 1) the map $T_\beta(x)=(\beta x) \mod 1$ on $[0,1)$ is equivalent to a ``shift" rightwards in the space of sequences $\{0,\dots,\beta-1\}^\mathbb{N}$ when $x$ is mapped to its base $\beta$ expansion; 2) $(x\beta^n)\mod 1=T_\beta^n(x)$.

When a non-integer base $\beta$ is considered, 2) is immediately false, while 1) has no clear reformulation, since there is no obvious candidate for a space of sequences that ``represent" numbers in base $\beta$. It is here that the theory of $\beta$-shifts and $\beta$-representations, developed, among others, by Parry~\cite{Parry60} and Bertrand~\cite{Bertrand86}, helps fill in the missing pieces.

Once the space of sequences that represent numbers in the base $\beta$ (using symbols from $\Sigma=\{0,\dots,\lceil\beta\rceil-1\}$) is defined, it is equipped with a natural shift transformation and a measure $P_\beta$ called the \textit{Parry measure}, which plays the same role that the uniform or Lebesgue measure played in integer representation. Indeed, a result by Bertrand says that, when $\beta$ is Pisot, if a real number $x$ has a $\beta$-expansion that is distributed according to $P_\beta$ (this is the analogue notion to being ``normal in base $\beta$''), then $(x\beta^n)_{n\in\NN}$ is u.d.\ modulo one.

To see how this is useful for the proof of Theorem~\ref{thm:main1}, let us say we have a number $z$ such that $(z\beta^n)_{n\in\NN}$ is not u.d.\ modulo one. Then, by Bertrand's theorem, its $\beta$-representation would have some block $\sigma$ whose average occurrences do not converge to $P_\beta(\sigma)$. We would then want to construct a polynomial time martingale that succeeds by betting on that block, as is done in the integer base case.

However, this cannot be done in a straightforward manner, since the martingale condition as used in the algorithmic randomness literature, assumes outcomes should be distributed according to the uniform measure.

We work with a generalized definition of martingales which captures the idea of a ``fair" betting strategy when expansions are supposed to obey some non-uniform distribution $P$. Indeed, this definition of a \textit{P-martingale} will capture the broader sense of \textit{martingale} as it is used in probability theory. In this setting, not only may the probability of the next symbol be different from $|\Sigma|^{-1}$, it may also show all forms of conditional dependence on the preceding symbols. It should be noted that randomness notions under measures different from Lebesgue have already been considered in, for example,~\cite{Reimann06}.

Schnorr and Stimm~\cite{Sch72} show that a sequence is normal in base $b$ if and only if no martingale on the alphabet of $b$ digits whose betting factors are computed by a deterministic finite automaton (DFA) succeeds on the expansion of $x$ in base $b$.
Our second main result is a generalization of this last statement in terms of $P$-martingales:
\begin{theorem}\label{thm:main2}
A sequence is distributed according to an irreducible, invariant Markov measure $P$ if an only if no $P$-martingale whose betting factors are computed by a DFA succeeds on it.
\end{theorem}

The importance of Markov measures is that they exhibit enough memorylessness to make them compatible with the memoryless structure of a DFA.

As regards $\beta$-representations, a second result by Bertrand establishes that for $\beta$ Pisot $P_\beta$, the natural measure on $\beta$-expansions, is ``hidden'' Markov. By extending Theorem~\ref{thm:main2} to hidden Markov measures we are able to construct a $P_\beta$-martingale generated by a DFA that succeeds on the $\beta$-expansion of $z$.
We use the polynomial time computability of the $\beta$-expansion and of the measure $P_\beta$ to show that an integer base (i.e.\ classical) martingale which succeeds on $z$ can be constructed from our $P_\beta$-martingale, following the same ideas used in~\cite{Fi2013}.

\subsection{Outline}

The paper is organized as follows. In \S\ref{sec:symdyn} we introduce some basics from symbolic dynamics, mainly the definition of Markov and sofic subshifts, and the notion of sequences distributed according to invariant measures $P$ over the shift. In \S\ref{sec:main} we introduce the notion of $P$-(super)martigales and show the characterization given by Theorem~\ref{thm:main2}. In \S\ref{sec:beta-exp-and-pisot} we introduce some definitions and results regarded to representation of reals in non-integer bases, in particular, Pisot bases. Finally, in \S\ref{sec:polyrndness} we put all pieces together to get Theorem~\ref{thm:main1}.

\section{Subshifts and measures}\label{sec:symdyn}

Throughout this work $\Sigma$ will denote an alphabet of finitely many symbols, which will be denoted by $a$, $b$, $c$, etc. The set of all words over the alphabet $\Sigma$ will be denoted by $\Sigma^*$, and the set of all words of length $k$ over the alphabet $\Sigma$ will be denoted $\Sigma^k$ (so $\Sigma^*=\bigcup_k\Sigma^k$).
Greek letters $\sigma$, $\tau$ and so on will be used for finite words in $\Sigma^*$
%
%in $\Sigma^*=\bigcup_k\Sigma^k$, the free monoid of words generated by $\Sigma$ (alternatively, $\Sigma^*$ may be regarded as the set of finite-length blocks in $\Sigma$).
Letters $s$, $s'$ will be used for infinite sequences in $\Sigma^\NN$. The $i$-th symbol of the sequence $s$ will be dented $s_i$.
%It should be noted that in the randomness literature $\Sigma$ is usually just $\{0,1\}$ and sequences in $\Sigma^\NN$ are regarded as sets of natural numbers and written $A$, $B$, etc. with entries $A(n)$, $A(n+1)$...
Concatenation will bear no special symbol, so we may write $\sigma=ab$, $s=as'$, $\rho=\sigma\tau$, etc. For a word $\sigma$ and $k\in\NN$ we denote with $\sigma^k$ the string of length $k|\sigma|$ which consists of the $k$-times repetition of $\sigma$, and with  $\sigma^\infty$ to the infinite sequence which consist of the repetition of $\sigma$ infinitely may times.
For any sequence $s\in\Sigma^\NN$ we will denote by $s\uh{N}$ the word that consists of the first $N$ symbols of $s$, and by $\langle s:k\rangle$ the same sequence $s$ when regarded as a sequence in $\Sigma^k$. For a word $\sigma$ and a non-negative integer $k$, let $\sigma\dh k$ denote the subword of $\sigma$ consisting of its last $k$ symbols (in case $l<k$ then $\sigma\dh k$ is just $\sigma$).
By $\sigma\preceq \tau$ we will denote that $\sigma$ is a prefix of
$\tau$, and by $\sigma\prec \tau$ we will denote that $\sigma$ is a
strict prefix of $\tau$. We will use the same notation ($\sigma\prec s$ and $\sigma\preceq s$) for sequences $s$.
 By $[\sigma]$ we will denote the cylinder set consisting of all infinite sequences extending $\sigma$, i.e.\ $[\sigma]=\{s\in\Sigma^\NN \colon \sigma\prec s\}$.

\begin{definition}
Given a finite alphabet $\Sigma$, a {\em subshift} is a tuple $(X,T)$ where
\begin{enumerate}
\item $X$ is some closed (hence, compact) subset of $\Sigma^\NN$ with the product topology
\item $X$ is invariant under $T$ (that is, $T(X)\subseteq X$); and
\item $T$ is the continuous mapping defined by $(T(s))_n=s_{n+1}$.
\end{enumerate}
If $X=\Sigma^{\NN}$ we say that $(X,T)$ is the full $|\Sigma|$-shift. The language associated to $(X,T)$, denoted $L(X)\subseteq \Sigma^*$, consists of all words appearing in the sequences of $X$.
\end{definition}

%\begin{definition}
%Given $(S,T)$, a subshift of $\Sigma^\NN$ (or $\Sigma^\ZZ$), its associated language $L(S)\subseteq \Sigma^*$ is the formal language which consists of all words appearing in the sequences of $S$.
%\end{definition}

Notice that $L(X)$ is a factorial and prolongable language, that is, it contains all subwords of its words and, if $\sigma\in L(X)$, then there exists a non-empty word $\tau$ in $\Sigma^*$ such that $\sigma\tau\in L(X)$.
Conversely, given any language, there is a corresponding closed subset of sequences.%
%\begin{definition}
For a language $L\subseteq\Sigma^*$ we define
$$
X_L=\{s\in\Sigma^\NN\colon  \forall N,\  s\uh{N}\ \in L\}.
$$
%\end{definition}
%
Observe that
$X_L$ is closed and that if $L$ is factorial then $X_L$ is shift invariant, hence, it is a subshift of $\Sigma^\NN$. Moreover, if $L$ is factorial and prolongable, then $L(X_L)=L$

\begin{definition}A $k$-{\em step Markov shift} (also known as a {\em subshift of finite type}, or {\em SFT}) is a subshift $(X,T)$ of $\Sigma^\NN$ such that there exists a set $G$ (called a \textit{grammar}) of admissible words of length $k$ satisfying
$$X=\{s\in\Sigma^A \colon (\forall i\in A)\ s_is_{i+1}\dots s_{i+k-1}\in G \}.$$
\end{definition}
These are called Markov shifts by analogy with the Markov processes of probability theory. For these, looking back at the last $k$ values of the process (say, the last $k$ flipped coins) is enough to know the probabilities of the next value (looking further backwards does not change these conditional probabilities). In the case of Markov shifts, looking at the last $k-1$ symbols is enough to know if the next symbol is admissible.

\begin{definition}
A probability measure $P$ on $\Sigma^\NN$ is called {\em $k$-step Markov} for some fixed $k\in \NN$ if for all $\sigma,\tau\in\Sigma^*$, $|\sigma|\geq k$,
%
%\begin{equation}\label{Markov}
$P([\sigma\tau]\mid [\sigma])=P([\rho\tau]\mid [\rho])$
%\end{equation}
where $\rho = \sigma\dh k$.
\end{definition}
The above condition is actually called $k$-step \textit{homogenous} Markov. A strict Markovian condition would read
$
P([\sigma\tau]\mid [\sigma])=P(T^{-(l-k)}([\rho\tau])\mid T^{-(l-k)}([\rho])).
$
 Since we will never consider non-homogenous Markov processes, we can spare the reader this extra terminology.
%
%While the right hand side expression of (\ref{Markov}) may look somewhat convoluted, it is just the probability of the word cylinder $[\rho\tau]$ occurring at position $l-k$ conditional on cylinder $[\rho]$ occurring at the same position.

From now on we will simplify notation and write $P(\sigma)$ instead of $P([\sigma])$ for any word $\sigma\in\Sigma^*$.
%
%\begin{definition}
Given a $1$-step Markov probability measure $P$ on $\Sigma^\NN$ we define its transition matrix $(p_{a,b})_{a,b\in \Sigma}$ to be
$$
p_{a,b}=P(ab\mid a).
$$
%\end{definition}
%
%\begin{definition}

An {\em invariant} measure on a subshift $(X,T)$ is a probability measure $P$ on $X$ (with its Borel $\sigma$-algebra $\mathcal{B}$) such that $P\circ T^{-1}=P$.
%\end{definition}
%
Notice that, by definition of $T$, $P\circ T^{-1}(\sigma)=\sum_{a\in\Sigma}P(a\sigma)$ for words $\sigma$, and invariance need only be checked for such word cylinders.

%\begin{observation}
Let $P$ be a 1-step Markov measure on $\Sigma^\NN$ with transition matrix $M=(p_{a,b})_{a,b\in\Sigma}$. Define the vector $v\in\RR^{\Sigma}$, $v_a=P(a)$. Then $P$ is invariant if and only if $v$ is a left eigenvector of $M$.
%\end{observation}
%
Let $P$ be a $k$-step Markov on $\Sigma^\NN$. Let $\Theta_k=\{\tau\in\Sigma^k\colon  P(\tau)>0\}$. Then $P$ induces a 1-step Markov measure $P^k$ on $\Theta_k^\NN$ with transition matrix $(p^k_{\sigma,\tau})_{\sigma,\tau\in\Theta_k}=P(\sigma\tau\mid\sigma)$.

%\begin{definition}
A probability measure $P$ on $\Sigma^\NN$ is called \textit{irreducible} if for any words $\sigma,\tau$ such that $P(\sigma)>0$, $P(\tau)>0$ there is some word $\rho$ such that $P(\sigma\rho\tau)>0$.
%\end{definition}
%
%\begin{observation}
A nonnegative $n\times n$ matrix $A$ is {\em irreducible} when the associated directed graph $G_A$, which has $n$ nodes and in which there is an edge from node $i$ to node $j$ if and only if $A_{ij}>0$, is strongly connected. A $1$-step Markov measure is irreducible if its transition matrix is irreducible, a $k$-step invariant Markov measure is irreducible if the matrix $(p^k_{\sigma,\tau})_{\sigma,\tau\in\Theta_k}$ is irreducible.
%\end{observation}

%
The following is the Perron-Frobenius Theorem for Markov chains in the finite state case (see \cite[Theorem 1.3.5]{Kitchens98} and \cite[Theorems 1.7.5-7 and Exercise 1.7.5]{Norr97}):
\begin{theorem}\label{perronfrobenius}Let $P$ and $P'$ be two invariant, irreducible 1-step Markov measures on $\Sigma^\NN$ such that their transition matrices are the same. Then $P=P'$.
\end{theorem}

%\begin{definition}
Given two subshifts $(X,T)$ and $(X',T)$ with $X\subseteq \Sigma^A,X'\subseteq \Sigma'^A$, a {\em factor map} is an onto map $\psi:X\rightarrow X'$ which commutes with the shift operator, that is, $\psi\circ T=T\circ \psi$.
%\end{definition}
Markov shifts are not closed under factor maps, but the following class of subshifts is.

\begin{definition}
A {\em sofic subshift} is the image of a Markov shift under a factor map.
\end{definition}

%\begin{theorem}[see Notes after section 3 of~\cite{Lind95}]
%A subshift $(S,T)$ with $S\subseteq \Sigma^A$ ($A=\{\NN,\ZZ\}$) is sofic if and only if $L(S)$ is a regular language (that is, it can be recognized by a finite automaton).
%\end{theorem}

\begin{example}\label{examplesofic}
Let us consider $X$ to be the 2-step Markov shift on $\{0,1\}^\NN$ with grammar $G=\{00,10,01\}$. For each $s$ in $X$, let $\psi(s)$ be such that $(\psi(s))_i=0$ if $s_i=s_{i+1}=0$ and $(\psi(s))_i=1$ otherwise. The image of $\psi$ is the set of infinite sequences such that all blocks of consecutive 1's are of even length (blocks of 0's are of arbitrary length). This corresponds to the regular expression $((11)^*0^*)^*$. Notice that this is not a Markov shift, since no matter how big $k$ is, looking back at the last $k$ values is not enough to determine whether a 0 is admissible next.
\end{example}

\begin{definition}
Given a subshift $(X,T)$, $s\in X$ and an invariant measure $\mu$ on $X$, we will say $s$ is {\em $\mu$-distributed} if for all continuous $f\colon X\rightarrow \RR$ we have
\begin{equation*}%\label{mudistributed}
\lim_{N\to\infty}\frac{\sum^{N-1}_{n=0}f(T^ns)}{N}= \int f d\mu.
\end{equation*}
\end{definition}
Notice that the above condition need only be checked on the characteristic functions of word cylinders (this is because characteristic functions of cylinders are dense in $C(\Sigma^\NN)$, since they form an algebra that separates points).  Then it is immediate that if $X=X_k$, the full $k$-shift for some integer $k>1$, and $\mu$ is the uniform or Lebesgue measure on $X$ with $\mu(i)=k^{-1}$ for $i\in\Sigma$, then $s$ is $\mu$-distributed if and only if the real number $\sum_{j>0}s_jk^{-j}$ is normal in base $k$.

There is a notion of entropy for dynamical systems called {\em metric entropy} or {\em Kolmogorov-Sinai entropy}, which is a natural extension of the Shannon entropy, and which assigns an entropy value $h_\mu(X)$ to any invariant measure $\mu^*$ on a system $X$. A measure $\mu^*$ has maximal entropy if $h_{\mu^*}(X)\geq h_\mu(X)$ for all invariant measures $\mu$ on $X$.
An important result concerning invariant measures for Markov shifts is the following, due to Parry~\cite{Parry64}:\footnote{An earlier and independent proof, in a somewhat different language, was already formulated in~\cite{Shannon48}.}

\begin{theorem}\label{teoparry2} Given an irreducible Markov shift $(X,T)$ with a grammar of wordlength $k-1$, there is a unique invariant probability measure $\tilde{P}$ on $X$ of maximal metric entropy.
Moreover, this measure is $k$-step Markov.
\end{theorem}

\section{$P$-martingales and $P$-distributed sequences}\label{sec:main}

In the algorithmic randomness literature, given a martingale $f$ on $\Sigma^*$,
one often constructs a (semi)measure
$
\mu_f(\sigma)=f(\sigma)|\Sigma|^{-|\sigma|},
$
which may be alternatively written as
\begin{equation}\label{martmeasure}
\mu_f(\sigma)=f(\sigma)\lambda([\sigma]),
\end{equation}
where $\lambda$ is the Lebesgue or uniform measure on $\Sigma^\NN$, which is taken to be the natural or ``fair'' measure on sequences of digits.

As was hinted in the introduction and by the mention of Theorem~\ref{teoparry2}, we will be interested in measures different from Lebesgue, i.e.\ we would like to substitute some arbitrary $P$ for $\lambda$ in the right hand side of~\eqref{martmeasure}. This forces us to change the definition of a martingale $f$, if we still want to make $\mu_f$ an additive measure.
%
%\begin{definition}
Given an alphabet $\Sigma$ and a language $L\subseteq \Sigma^*$, a probability measure $P$ on $\Sigma^\NN$ is called \textit{L-supported} if
$
P(\sigma)=0\ \Leftrightarrow\ \sigma\in\Sigma^*\setminus L.
$
Equivalently, $P$ has full support on $X_L$.
%\end{definition}
%
%
\begin{definition}\label{def:supermartingale}Given an alphabet $\Sigma$, a language $L\subseteq \Sigma^*$ and some $L$-supported probability measure $P$ on $\Sigma^\NN$, a {\em $P$-supermartingale} on $L$ is a function $f\colon L\rightarrow \RR$ satisfying
\begin{equation}\label{eqn:supermartingale-condition}
f(\sigma)\geq\sum_{\substack{a\in\Sigma\\ \sigma a\in L}}P(\sigma a \mid \sigma)f(\sigma a).
\end{equation}
for all $\sigma$ in $L$.
The function $f$ is called a {\em $P$-martingale} if the above inequality can be replaced by an equality for all $\sigma\in L$.
We say that $f$ {\em succeeds} on $s\in\Sigma^\NN$ if
$\limsup_Nf(s\uh{N})=\infty$. The ratios $f(\sigma a)/f(\sigma)$ are called \textit{betting factors} of $f$.
\end{definition}
Notice that the conditional probabilities in~\eqref{eqn:supermartingale-condition} are always well-defined since $P$ is $L$-supported and $\sigma\in L$.
Of course, when $P$ is $\lambda$ as in Definition~\ref{def:supermartingale}, the classical definition of a martingale is recovered, since $\lambda(\sigma a\mid \sigma)=\lambda(a)=|\Sigma|^{-1}$. This generalized definition is somewhat more intuitive in the sense that it makes explicit the real-life fact that the odds offered by a bookie at some gamble are the inverse of some implied probability (conditional on the available information) on the outcomes of the gamble.
%\footnote{In real-life gambles the odds are a little bit less than that since the bookie must make a living. $P$-supermartingales are thus better suited to modelling real-life situations.}
Classical martingales then just capture the case when these probabilities are uniform and independent of previous outcomes.

%\begin{definition}Given an alphabet $\Sigma$, a measure $P$ on $\Sigma^\NN$, a $P$-supermartingale $f$ and a sequence $s\in \Sigma^\NN$ we say that $f$ {\em succeeds} on $s$ if
%$\limsup_Nf(s\uh{N})=\infty$.
%\end{definition}

We define now the notion of $P$-martingale generated by a deterministic finite automaton (DFA). This is a generalization of the notion of a classical betting strategy generated by a DFA, introduced in~\cite{Sch72}.
We will write automata in the usual form $M=\langle Q,\Sigma,\delta,q_0,Q_f\rangle$, where $Q$ is a finite set of states, $\Sigma$ is the input alphabet, $\delta$ is the transition function, $q_0$ is the initial state and $Q_f\subseteq Q$ is the set of accepting states. Also, we will use the notation $\delta^*$ for the natural extension of the transition function $\delta$ from symbols to words in $\Sigma$.
\begin{definition}A $P$-martingale $f$ on a language $L$ is \textit{generated by a DFA}  if there is a DFA $M$ accepting $L$, and a function $b\colon Q\times \Sigma\rightarrow \RR$ such that
$$
f(\sigma a)=b(\delta^*(\sigma,q_0),a)f(\sigma)
$$
for any word $\sigma$ and symbol $a$ such that $\sigma a\in L$.
\end{definition}

The main result of this section is that any sequence is distributed according to an irreducible, invariant Markov measure $P$ if an only if no $P$-martingale generated by a DFA succeeds on it. The rest of the section is devoted to show it. In \S\ref{sec:mart-can-beat} we show the `if' implication and in \S\ref{sec:mart-cannot-beat} we show the `only if' implication. For the case of $P$ being a measure on a sofic shift, we extend the `if' direction in \S\ref{sec:extension-to-sofic}. This generalization will be needed for \S\ref{sec:polyrndness}.

\subsection{$P$-martingales on a DFA can beat sequences that are not $P$-distributed}\label{sec:mart-can-beat}

\begin{theorem}\label{maintheorem}
Let $\Sigma$ be an alphabet, $(X,T)$ a subshift of $\Sigma^\NN$, and let $P$ be a $L(X)$-supported $k$-step Markov invariant measure on $\Sigma^\NN$ such that $(p_{\sigma,\tau}^k)_{\sigma,\tau\in\Theta_k}$, the Markov transition matrix induced on $\Theta_k^\NN$, is irreducible. Suppose $s\in X$ is not $P$-distributed. Then there is a $P$-martingale generated by a DFA which succeeds on $s$.
Moreover, the only betting factors of this martingale are 1, $(1+\delta)$ and $(1-\delta p^*/(1-p^*))$, where $\delta$ is rational and $p^*=P(\tau\rho\mid\tau)$ or $p^*=1-P(\tau\rho\mid\tau)$ for some $\tau,\rho\in\Sigma^*$.
\end{theorem}

Before proceeding to the proof of the theorem we present some useful notation and auxiliary lemmas.
For words $\sigma,\tau\in \Sigma^*$, we let $\occ(\tau,\sigma)$ be the number of occurrences of $\tau$ in $\sigma$, that is
$$
\occ(\tau,\sigma)=\vert\{i\colon  0\leq i\leq |\sigma|-|\tau|,\tau=\sigma_{i}\dots\sigma_{i+|\tau|-1}\}\vert.
$$

For $k$ an integer, $P$ a measure on $\Sigma^\NN$, $\sigma\in\Sigma^*$ and $A\subseteq \Sigma^*$ we write
%\begin{eqnarray*}
$$
\Prec_k(\sigma)=\{\tau\in\Sigma^k \colon \ P(\tau\sigma)>0\}, \mbox{\qquad and\qquad}
\Prec_k(A)=\bigcup_{\sigma\in A}\Prec_k(\sigma).
$$
%\end{eqnarray*}
%
We will also make use of the following functions $M_s$ and $m_s$ defined on $\Sigma^*$
%
%\begin{eqnarray*}
$$
M_s(\sigma)=\limsup_{N\to\infty}\frac{\occ(\sigma,s\uh{N})}{N},\mbox{\qquad and\qquad}
m_s(\sigma)=\liminf_{N\to\infty}\frac{\occ(\sigma,s\uh{N})}{N},
$$
%\end{eqnarray*}
for some fixed $s\in\Sigma^\NN$.
%
%Notice that this function is monotonic: if $\sigma$ is a subword of $\tau$ then $M_s(\tau)\leq M_s(\sigma)$. It is also $\textit{superadditive}$ in a certain sense, that is, for any $m$,
%\begin{align}\label{superadd1}
%\sum_{\sigma\in\Sigma^m}M_s(\sigma\tau)&=\sum_{\sigma\in\Sigma^m}\limsup_{N\to\infty}\frac{\occ(\sigma\tau,s\uh{N})}{N}\geq\limsup_{N\to\infty}\frac{\sum_{\sigma\in\Sigma^m}\occ(\sigma\tau,s\uh{N})}{N}=\nonumber\\
%&=\limsup_{N\to\infty}\frac{\occ(\tau,s\uh{N})}{N}=M_s(\tau)
%\end{align}
%
%And similarly,
%\begin{equation}\label{superadd2}
%\sum_{\sigma\in\Sigma^m}M_s(\tau\sigma)\geq M_s(\tau)
%\end{equation}
%
The subscript $s$ will often be omitted from $M_s$ and $m_s$ when it is understood from context.

For the sake of simplicity, since the step $k$ is fixed, we will write $p_{\sigma,\tau}=p_{\sigma,\tau}^k$. To prove Theorem~\ref{maintheorem} we will first need some auxiliary lemmas. For the rest of this section, the measure $P$ is assumed to satisfy the conditions of Theorem~\ref{maintheorem}.

\begin{lemma}\label{lemma1} Suppose $s\in X$ is not $P$-distributed and that
$m_s(\tau^*)\neq 0$
for some $\tau^*\in\Theta_k$.
Then there is some $\sigma^*\in L(X)$ with $|\sigma^*|\geq k$ and $b\in\Sigma$  such that%\fxnote{$\sigma^* b\in L(X)$?}
\begin{equation}\label{cond1}
\frac{\occ(\sigma^* b,s\uh{N})}{\occ(\sigma^*,s\uh{N})}\not\to P(\sigma^* b\mid \sigma^*).
\end{equation}
when $N\to\infty$. Moreover, $\sigma^*$ can be chosen so that $
%\begin{equation}\label{satisfaceliminf}
m_s(\sigma^*)>0.
%\end{equation}
$
\end{lemma}

\begin{proof}
Let $s\in X$ not be $P$-distributed, and let $M=M_s$ and $m=m_s$.
Let us define
$$r_{\sigma,\tau}=\lim_{N\to\infty}\frac{\occ(\sigma\tau,s\uh{N})}{\occ(\sigma,s\uh{N})}$$
for any words $\sigma,\tau\in\Sigma^*$, whenever the limit exists.

The proof follows by contradiction, so let us assume that for all words $\sigma$  with $|\sigma|\geq k$ and $P(\sigma)>0$, and any $b\in\Sigma$ we have
%
%\begin{equation}\label{eq}
$
r_{\sigma,b}= P(\sigma b\mid \sigma).
$
%\end{equation}

%there is no $\sigma^* b\in L(X)$ with $|\sigma^* b|>k$ satisfying (\ref{cond1}).

\medskip

%\noindent {\bf Step 1.}
%{\em We extend~\eqref{eq} from the conditional %probability of the next symbol to the %conditional probability of any following word, %i.e.\ we show
%; that is, from
%\begin{equation}\label{eq}
%r_{\sigma,b}= P(\sigma b\mid \sigma)
%\end{equation}
%for all words $\sigma$ with $|\sigma|\geq k$, %$P(\sigma)>0$, and any symbol $b$, we will prove

\begin{proposition}\label{prop:lemma1-step1}
For all words $\sigma$ with $|\sigma|\geq k$, $P(\sigma)>0$, and any word $\tau=b_1\dots b_m\in\Sigma^*$ we have
%\begin{equation}\label{eq2}
$
r_{\sigma,\tau}= P(\sigma \tau\mid \sigma).
$
%\end{equation}
\end{proposition}

\begin{proof}
Given $\tau$, we first take the largest $j$ such that $P(\sigma b_1\dots b_{j-1})>0$. Then, by an iterated use of $r_{\sigma,b}= P(\sigma b\mid \sigma)$ we get
\begin{eqnarray*}
r_{\sigma, b_1\dots b_j} &=& \prod_{i=1}^{j-1} r_{\sigma b_1\dots b_i, b_{i+1}}=\prod_{i=1}^{j-1} P(\sigma b_1\dots b_{i+1}\mid \sigma b_1\dots b_i)=P(\sigma b_1\dots b_j\mid \sigma)
\end{eqnarray*}
If $j=m$ we are done. %this proves (\ref{eq2}).
Otherwise, we have
$$0=P(\sigma b_1\dots b_{j+1})=P(\sigma b_1\dots b_{j+1}\mid \sigma b_1\dots b_j)=r_{\sigma b_1\dots b_j, b_{j+1}}.$$
Notice that
$$0=r_{\sigma b_1\dots b_j, b_{j+1}}=\lim_{N\to\infty}\frac{\occ(\sigma b_1\dots b_{j+1},s\uh{N})}{\occ(\sigma b_1\dots b_j,s\uh{N})}\geq \limsup_{N\to\infty}\frac{\occ(\sigma b_1\dots b_m,s\uh{N})}{\occ(\sigma,s\uh{N})},$$
and hence $=r_{\sigma,b_1\dots b_m}$ exists and is equal to 0. Then
 $$r_{\sigma, b_1\dots b_m}=0=P(\sigma b_1\dots b_{j+1})\geq P(\sigma b_1\dots b_m)\geq 0$$
implies $P(\sigma\tau\mid\sigma)=0$, which finishes our proof of Proposition~\ref{prop:lemma1-step1}.
\end{proof}

\begin{proposition}\label{prop:lemma1-step2}
For any $\tau\in\Theta_k$ we have
\begin{equation}\label{limdem}
M(\tau)=\lim_{N\to\infty}\frac{\occ(\tau,s\uh{N})}{N}.
\end{equation}
\end{proposition}
\begin{proof}
Fix some $\tau^*\in\Theta_k$ such that $m(\tau^*)\neq 0$ and let $\tau_1\dots\tau_\ell$ be an enumeration of all the other words in $\Theta_k$. It should be noted that if we define $\Theta_k^M=\{\tau\in\Sigma^k\mid\ M(\tau)>0\}$, then from Proposition~\ref{prop:lemma1-step1} and the fact that $P$ is $L(X)$-supported and $s\in X$ it is easy to deduce that $\Theta_k^M\subseteq\Theta_k$. This fact will be implicit in the following calculations. Then, for any $i\leq \ell$,
\begin{align}\label{eqgrosa}
\limsup_{N\to\infty}\frac{\occ(\tau_i,s\uh{N})}{\occ(\tau^*,s\uh{N})}&=\limsup_{N\to\infty}\sum_{\tau\in\Theta_k}\frac{\occ(\tau\tau_i,s\uh{N})}{\occ(\tau^*,s\uh{N})}\nonumber\\
&=p_{\tau^*,\tau_i}+\limsup_{N\to\infty}\sum_{j=1}^\ell\frac{\occ(\tau_j\tau_i,s\uh{N})}{\occ(\tau_j,s\uh{N})}\frac{\occ(\tau_j,s\uh{N})}{\occ(\tau^*,s\uh{N})}\nonumber\\
&\leq p_{\tau^*,\tau_i}+\sum_{j=1}^\ell p_{\tau_j,\tau_i}\limsup_{N\to\infty}\frac{\occ(\tau_j,s\uh{N})}{\occ(\tau^*,s\uh{N})}.
\end{align}

Notice that
$$\limsup_{N\to\infty}\frac{\occ(\tau_i,s\uh{N})}{\occ(\tau^*,s\uh{N})}\leq\limsup_{N\to\infty}\frac{\occ(\tau_i,s\uh{N})}{N}\left(\liminf_{N\to\infty}\frac{\occ(\tau^*,s\uh{N})}{N}\right)^{-1}=\frac{M(\tau_i)}{m(\tau^*)}<\infty.$$
Hence can write
$$x_i=\limsup_{N\to\infty}\frac{\occ(\tau_i,s\uh{N})}{\occ(\tau^*,s\uh{N})}$$
and $x=(x_1,\dots,x_\ell)$, and reformulate~\eqref{eqgrosa} in matrix form as follows
\begin{equation}\label{eqmatrixsup}
(\textbf{id}-R^*)x\leq p^*,
\end{equation}
where $\leq$ is the product order on $\RR^\ell$, $R^*$ is the transpose of the Markov transition matrix $p_{\sigma,\tau}$ restricted to $\Theta_k\setminus\{\tau^*\}$ and $p^*=(p_{\tau^*,\tau_1},\dots,p_{\tau^*,\tau_\ell})$.

Similarly, if
$$y_i=\liminf_{N\to\infty}\frac{\occ(\tau_i,s\uh{N})}{\occ(\tau^*,s\uh{N})}$$
and $y=(y_1,\dots,y_\ell)$, then the same reasoning used in~\eqref{eqgrosa} shows
\begin{equation}\label{eqmatrixinf}
(\textbf{id}-R^*)y\geq p^*.
\end{equation}

Let us write $A=(\textbf{id}-R^*)$ and show that $Ax=Ay$. Equations~\eqref{eqmatrixsup} and~\eqref{eqmatrixinf} imply $Ax\leq Ay$, so it suffices to show that a contradiction follows from assuming $Ax<Ay$. Indeed, $Ax<Ay$ means that, for all $i$, $\sum_j A_{ij}x_j\leq \sum_j A_{ij}y_j$, where the inequality is strict for some $i$. This, in turn, implies
$$\sum_j\left(\sum_iA_{ij}\right)x_j=\sum_i\sum_jA_{ij}x_j<\sum_i\sum_jA_{ij}y_j=\sum_j\left(\sum_iA_{ij}\right)y_j,$$
which is impossible since $x_j\geq y_j$ for all $j$ and $\sum_i A_{ij}=1-\sum_i p_{\tau_j,\tau_i}\geq 0$. Thus, $Ax=Ay$.

Now, if $A$ were invertible then it would follow that $x=y$, which means $\lim_{N\to\infty}\ \occ(\tau_i,s\uh{N})/\occ(\tau^*,s\uh{N})$ exists for all $i$, and this in turn implies that~\eqref{limdem} is true for $\tau^*$, that is, $\occ(\tau^*,s\uh{N})/N$ converges (to $M(\tau^*)$, its $\limsup$), since
$$\sum_i\lim_{N\to\infty}\frac{\occ(\tau_i,s\uh{N})}{\occ(\tau^*,s\uh{N})}+1=\lim_{N\to\infty}\frac{N}{\occ(\tau^*,s\uh{N})}$$
and from this convergence for $\tau^*$ we derive that of $\tau_i$ for all $i$ using
$$\lim_{N\to\infty}\frac{\occ(\tau_i,s\uh{N})}{\occ(\tau^*,s\uh{N})}\lim_{N\to\infty}\occ(\tau^*,s\uh{N})=\lim_{N\to\infty}\frac{\occ(\tau_i,s\uh{N})}{N}.$$

So it remains to show that $A$ is indeed invertible. If it were not, then $R^*$ would have 1 as an eigenvalue, and the Perron-Frobenius theorem, together with the fact that the column sums of $R^*$ are smaller than 1, imply 1 has a unique nonnegative eigenvector $z=(z_1,\dots,z_\ell)$. That is,
\begin{equation}\label{eqinve}
\sum_{i=1}^\ell z_i R^*_{ji}=\sum_{i=1}^\ell z_ip_{\tau_i,\tau_j}=z_j.
\end{equation}

Now, $\tau^*\in\Theta_k$ is excluded from the enumeration $(\tau_i)_{1\leq i\leq \ell}$ and the irreducibility of the matrix $p_{\sigma,\tau}$ implies that $\Prec_k(\tau^*)\cap\left(\Theta_k\setminus\{\tau^*\}\right)$ is not empty.
Hence, there is some $\tau_i\in\Prec_k(\tau^*)$ and for each such $i$ we have $\sum_{j=1}^\ell p_{\tau_i,\tau_j}<1$, so that if $z_i\neq 0$ then~\eqref{eqinve} implies
$$\sum_{j=1}^\ell z_j=\sum_{i,j=1}^\ell z_ip_{\tau_i,\tau_j}=\sum_{i=1}^\ell z_i\sum_{j=1}^\ell p_{\tau_i,\tau_j}<\sum_{i=1}^\ell  z_i,$$
which is a contradiction.

Thus, $z_i=0$ for all $i$ such that $\tau_i\in\Prec_k(\tau^*)$. This in turn implies $p_{\tau_j,\tau_i}=0$ for all $j$ such that $z_j \neq 0$, since $0=z_i=\sum_j z_jp_{\tau_j,\tau_i}$. Equivalently, $z_j=0$ for all $j$ such that $\tau_j\in\Prec_k(\tau^i)$.

We then repeat this reasoning to show $z_k=0$ for all $k$ such that $\tau_k\in\Prec_k(\tau^j)$ and keep repeating the same reasoning until all entries in $z$ have been shown to be 0 (this is guaranteed by irreducibility). Hence, $z=0$, which contradicts the assumption that $z$ is an eigenvector of eigenvalue 1. It follows that $A$ must be invertible. This concludes the proof of Proposition~\ref{prop:lemma1-step2}.
\end{proof}

\begin{proposition}\label{prop:M-is-P}
$M$ is equal to $P$ restricted to $\Theta_k$.
\end{proposition}

\begin{proof}
Now, $M$ is a probability measure on $\Theta_k$, since
$$
\sum_{\tau\in\Theta_k}M(\tau)=\sum_{\tau\in\Theta_k}\lim_{N\to\infty}\frac{\occ(\tau,s\uh{N})}{N}=\lim_{N\to\infty}\frac{\sum_{\tau\in\Theta_k}\occ(\tau,s\uh{N})}{N}=1.
$$
Together with the Markov transition matrix $p_{\sigma,\tau}$, $M$ defines a probability measure $\nu$ on $\Theta_k^\NN$ in a natural way. First, $\nu$ is defined inductively on word cylinders
\begin{eqnarray*}
\nu([\tau])&=&M(\tau)\\
\nu([\tau_1\dots \tau_j])&=&\nu([\tau_1\dots\tau_{j-1}])p_{\tau_{j-1},\tau_j}
\end{eqnarray*}
then extended naturally to all cylinders and finally to the Borel $\sigma$-algebra $\mathcal{B}$ of $\Theta_k^\NN$ via Caratheodory's extension theorem.

As with $P$, we will drop the brackets for word cylinders.
To show that $\nu$ is invariant, it is enough to show it for word cylinders, that is, it is enough to show
\begin{equation*}%\label{invariant}
\nu(T^{-1}(\tau_1\dots\tau_j))=\sum_{\sigma\in\Theta_k}\nu(\sigma\tau_1\dots\tau_j)=\nu(\tau_1\dots\tau_j).
\end{equation*}
By our construction of $\nu$,
\begin{align*}
\sum_{\sigma\in\Theta_k}\nu(\sigma\tau_1\dots\tau_j)&=\sum_{\sigma\in\Theta_k}\nu(\sigma)\nu(\tau_1\mid\sigma)\prod_{i=1}^{j-1}\nu(\tau_{i+1}\mid\tau_i)\\
&=\sum_{\sigma\in\Theta_k}M(\sigma)p_{\sigma,\tau_1}\prod_{i=1}^{j-1}p_{\tau_i,\tau_{i+1}}=\left(\prod_{i=1}^{j-1}p_{\tau_i,\tau_{i+1}}\right)\sum_{\sigma\in\Theta_k}M(\sigma)r_{\sigma,\tau_1}\\
&=\left(\prod_{i=1}^{j-1}p_{\tau_i,\tau_{i+1}}\right)\sum_{\sigma\in\Theta_k}\lim_{N\to\infty}\frac{\occ(\sigma,s\uh{N})}{N}\lim_{N\to\infty}\frac{\occ(\sigma\tau_1,s\uh{N})}{\occ(\sigma,s\uh{N})}\\
&=\left(\prod_{i=1}^{j-1}p_{\tau_i,\tau_{i+1}}\right)\sum_{\sigma\in\Theta_k}\lim_{N\to\infty}\frac{\occ(\sigma\tau_1,s\uh{N})}{N}\\
&=\left(\prod_{i=1}^{j-1}p_{\tau_i,\tau_{i+1}}\right)\lim_{N\to\infty}\frac{\occ(\tau_1,s\uh{N})}{N}\\&=\left(\prod_{i=1}^{j-1}p_{\tau_i,\tau_{i+1}}\right)M(\tau_1)=\nu(\tau_1\dots\tau_j)\addtag\label{eqlarga}
\end{align*}
Thus, $\nu$ is invariant, 1-step Markov and has the irreducible Markov transition matrix $p_{\sigma,\tau}$. Theorem~\ref{perronfrobenius} implies that $P=\nu$ and $M$ is equal to $P$ restricted to $\Theta_k$, and this concludes the proof of Proposition~\ref{prop:M-is-P}.
\end{proof}

We finally show that $s$ is $P$-distributed, leading to a contradiction.
In~\eqref{eqlarga} we show that
$$\nu(\sigma\tau)=M(\sigma)p_{\sigma,\tau}=P(\sigma)r_{\sigma,\tau}=\lim_{N\to\infty}\frac{\occ(\sigma\tau,s\uh{N})}{N}$$
for $\tau\in\Theta_k$. This extends trivially to $\tau\in\Sigma^k$, for $M(\tau)=0$ if and only if $P(\tau)=0$ and $P=\nu$. Moreover, the same is valid if we substitute any word $\rho$ for $\tau$ in the above equations, since all we need is that $r_{\sigma,\rho}$ exist and be equal to $P(\sigma\rho\mid\sigma)$. Since $\sigma$ must be of length $k$, this means that
\begin{equation}\label{pdistri}
\lim_{N\to\infty}\frac{\occ(\rho,s\uh{N})}{N}=\nu(\rho)=P(\rho)
\end{equation}
for all words $\rho$ of length at least $k$.
But then~\eqref{pdistri} must also be true for words $\rho$ of length smaller than $k$, since
$$
P(\rho)=\sum_{\substack{\tau\in\Theta_k\\ \rho\prec\tau}}P(\tau)=\sum_{\substack{\tau\in\Theta_k\\ \rho\prec\tau}}\lim_{N\to\infty}\frac{\occ(\tau,s\uh{N})}{N}=\lim_{N\to\infty}\frac{\sum_{\substack{\tau\in\Theta_k\\ \rho\prec\tau}}\occ(\tau,s\uh{N})}{N}=\lim_{N\to\infty}\frac{\occ(\rho,s\uh{N})}{N}.
$$
$P$-distribution need only be checked on word cylinders, so this completes the proof that some $\sigma^*$ satisfies~\eqref{cond1}.

It only remains to show that such a $\sigma^*$ can be chosen so that $m_s(\sigma^*)>0$.
Again, we prove this by contradiction. That is, let us suppose that for all $\sigma\in L(X)$ ($|\sigma|\geq k$) such that $m(\sigma)>0$, we have that $r_{\sigma,b}$ exists for any symbol $b$ and is equal to $P(\sigma b\mid\sigma)$. As before, this implies $r_{\sigma,\tau}$ exists for all words $\tau$ and is equal to $P(\sigma\tau\mid\sigma)$.

Take some $\sigma^*$ that satisfies~\eqref{cond1}. Then $m(\sigma^*)=0$. Take some $\tau$ such that $P(\tau\sigma^*)>0$ (irreducibility implies this can be done by finding some $(\tau_i)_{1\leq i\leq l}\subseteq\Theta_k$ such that $\sigma^*\prec\tau_1\dots\tau_l\in L(X)$ and then finding some $\tau\in\Prec_k(\tau_i)$). If $m(\tau)>0$ then $r_{\tau,\sigma^*}$ exists and is equal to $P(\tau\sigma^*\mid\tau)>0$. But this contradicts the fact that $m(\tau\sigma^*)\leq m(\sigma^*)=0$. So $m(\tau)=0$ for all $\tau\in\Prec_k(\sigma^*)$.

Similarly, for all $\sigma\in\Prec_k(\Prec_k(\sigma^*))$ we have $m(\sigma)=0$ and the same reasoning can be repeated until $m(\sigma)=0$ has been shown for all $\sigma\in\Theta$ (irreducibility guarantees this), which contradicts the condition that $m(\tau^*)>0$ for some $\tau^*\in\Theta_k$. This concludes the proof of Lemma~\ref{lemma1}.
\end{proof}

\begin{lemma}\label{lemma2}Given $s\in X$, if there is some $\sigma^*\in\Theta_k$ that satisfies $m(\sigma^*)=0$, then there are some $d>0$, $\rho\in\Theta_k$, $\sigma\in\Prec_k(\rho)$ and a strictly increasing sequence $(N_j)_{j\in \NN}$ of natural numbers such that
$
\lim_{j\to\infty} \occ(\rho,s\uh{N_j})/N_j=0$ and
$
\limsup_{j\to\infty} \sum_{\tau\in\Theta_k\setminus\{\rho\}} \occ(\sigma\tau,s\uh{N_j})/N_j\geq d.
$
\end{lemma}

\begin{proof} Let $\sigma^*\in\Theta_k$ satisfy $m(\sigma^*)=0$ and let $(N_j)_{j\in\NN}$ be a strictly increasing sequence of natural numbers such that
$
\lim_{j\to\infty} \occ(\sigma^*,s\uh{N_j})/N_j=0.
$
If for some $\sigma\in\Prec_k(\sigma^*)$ and some $d>0$ we have
$\limsup_{j\to\infty} \sum_{\tau\in\Theta_k\setminus\{\sigma^*\}} \occ(\sigma\tau,s\uh{N_j})/N_j\geq d,
$
then we set $\rho=\sigma^*$ and we are done.

Otherwise, we have, for $\epsilon>0$, a $j_0$ such that for all $j\geq j_0$,
$$
\sum_{\substack{\sigma\in\Prec_k(\sigma^*)\\ \tau\in\Theta_k\setminus\{\sigma^*\}}}\frac{\occ(\sigma\tau,s\uh{N_j})}{N_j}<\frac{\epsilon}{|\Theta_k|},
$$
and since
$\lim_{j\to\infty} \occ(\sigma\sigma^*,s\uh{N_j})/N_j\leq\lim_{j\to\infty} \occ(\sigma^*,s\uh{N_j})/N_j=0,$
we conclude, for all $j$ greater than some $j_0$,
$$
\frac{\sum_{\sigma\in\Prec_k(\sigma^*)} \occ(\sigma,s\uh{N_j})+\occ(\sigma^*,s\uh{N_j})}{N_j}<2\epsilon.
$$
Hence, if we write $B_0=\{\sigma^*\}$, $B_{t+1}=B_t\cup\Prec_k(B_t)$ and, for any finite $A\subseteq \Sigma^*$
$$
\occ(A)=\limsup_{j\to\infty}\frac{\sum_{\sigma\in A}\occ(\sigma,s\uh{N_j})}{N_j},
$$
then we have just shown that $\occ(B_1)=\occ(\Prec_k(\sigma^*))=0$.

Similarly, given $t$ such that $\occ(B_t)=0$
we can use the same reasoning to show that either
$\limsup_{j\to\infty} \sum_{\tau\in\Theta_k\setminus\{\rho\}} \occ(\sigma\tau,s\uh{N_j})/N_j\geq d
$
for some $d>0$, $\rho\in B_t$ and some $\sigma\in \Prec_k(\rho)\subseteq B_{t+1}$, in which case we are done, or else $\occ(B_{t+1})=0$.

But irreducibility implies that, for some $p$, $B_p=\Theta_k$ and we cannot have $0=\occ(B_p)=\occ(\Theta_k)=1$. Hence, there is some $t$ and some $d>0$, $\rho\in B_t$ and $\sigma\in \Prec_k(\rho)$ such that
$
\lim_{j\to\infty} \occ(\rho,s\uh{N_j})/N_j\leq\lim_{j\to\infty} \sum_{\sigma\in B_t} \occ(\sigma,s\uh{N_j})/N_j
$
and
$$
\limsup_{j\to\infty} \sum_{\tau\in\Theta_k\setminus\{\rho\}} \occ(\sigma\tau,s\uh{N_j})/N_j\geq d.
$$
This concludes the proof of Lemma~\ref{lemma2}.
\end{proof}

%Irreducibility of the Markov transition matrix when $P$ is regarded as a measure on $\Theta_k^\NN$  means that blocks of $k$-words and positive probability are strongly connected in the sense that, for $\sigma,\tau\in\left(\Theta_k\right)^*$, there is some $\rho\in\left(\Sigma^k\right)^*$ such that $P(\sigma\rho\tau)>0$. Informally speaking, we can go from $\sigma$ to $\tau$ with positive probability using $\rho$.

%Therefore, $P(\alpha'\rho\beta')>0$. And since $M(\beta')=0$, then $M(\rho\beta')=0$ and $\beta=\rho\beta'\in C$.

%Since $M(\beta)=0$, it follows that $\liminf\occ(\beta,s\uh{N})/N=0$ and we are now under the conditions of case I.

%\end{proof}

\begin{proof}[Proof of Theorem~\ref{maintheorem}]We will split our proof in two cases.

\paragraph{Case I.} {\em There is some $\tau^*\in \Theta_k$ such that $m(\tau^*)>0$.}

From Lemma~\ref{lemma1} we may assume that for some $\sigma\in L(X)$ satisfying $|\sigma|\geq k$ and $m(\sigma)>0$, some $b\in\Sigma$ and some rational $\delta>0$
\begin{equation}\label{limsup}
\limsup_{N\to\infty}\frac{\occ(\sigma b,s\uh{N})}{\occ(\sigma,s\uh{N})}>(1+\delta)P(\sigma b\mid\sigma),
\end{equation}
since~\eqref{cond1} implies either~\eqref{limsup} or
$$
\liminf_{N\to\infty}\frac{\occ(\sigma b,s\uh{N})}{\occ(\sigma,s\uh{N})}<(1-\delta)P(\sigma b\mid\sigma),
$$
but in the latter case it is easy to find some $b'$ such that~\eqref{limsup} is true for $\sigma b'$.

We define our $P$-martingale $L$ by:
\begin{align*}
L(\emptyset)&=1\\
  L(\rho c)&=\begin{cases}
    (1+\delta) L(\rho) & \text{if $\rho\dh{ |\sigma|}=\sigma$ and $c=b$};\\
    \left(1-\frac{\delta p^*}{1-p^*}\right)L(\rho) & \text{if $\rho\dh{|\sigma|}=\sigma$ and $c\neq b$};\\
    L(\rho) & \text{otherwise};
  \end{cases}
\end{align*}
for any $c\in\Sigma$, $\rho c\in L(X)$,
where
$p^*=P(\sigma b|\sigma)$ and we further impose that $\delta<(1-p^*)/p^*$.
Notice that for all $\rho$ such that $\rho\dh{|\sigma|}=\sigma$ the $k$-step Markov property and $|\sigma|\geq k$ impliy that $p^*=P(\rho b|\rho)$. From this it is easy to see that $L$ is a $P$-martingale, and it is also clearly generated by a DFA, since at each step the betting factor depends solely on the next symbol and the previous $|\sigma|$ symbols of $\rho$ and since there are finitely many words of length $\sigma$, it suffices to consider the finite set of states $Q=|\Sigma|^{|\sigma|}$.

To see that $L$ succeeds on $s$, we observe first that
$$L
(\rho)=(1+\delta)^{\occ(\sigma b,\rho)}\prod_{c\neq b}\left(1-\frac{\delta p^*}{1-p^*}\right)^{\occ(\sigma c,\rho)}
$$
and that
\begin{equation}\label{limocc}
\sum_{c\neq b}\occ(\sigma c,s\uh{N})\leq\occ(\sigma,s\uh{N})-\occ(\sigma b,s\uh{N}).
\end{equation}
Let $r=1+\delta$ and $q=1-\frac{\delta p^*}{1-p^*}$. Equation~\eqref{limocc} then implies
\begin{align}\limsup_{N\to\infty} \frac{\log L(s\uh{N})}{N} &\geq\limsup_{N\to\infty} \frac{\occ(\sigma b,s\uh{N})}{N}(\log r-\log q)+\frac{\occ(\sigma,s\uh{N})}{N}\log q \nonumber\\
& =\limsup_{N\to\infty} \frac{\occ(\sigma,s\uh{N})}{N}\left[\frac{\occ(\sigma b,s\uh{N})}{\occ(\sigma,s\uh{N})}(\log r-\log q)+\log q\right]\nonumber\\
& \geq\liminf_{N\to\infty} \frac{\occ(\sigma,s\uh{N})}{N}\limsup_{N\to\infty}\left[r\frac{\occ(\sigma b,s\uh{N})}{\occ(\sigma,s\uh{N})}(\log r-\log q)+\log q\right]\nonumber\\
& \geq m(\sigma)\left[rP(\sigma b\mid\sigma)(\log r-\log q)+\log q\right] \nonumber\\
&= p^*m(\sigma)\left[(1+\delta)\log (1+\sigma)+\left(p^{*-1}-(1+\sigma)\right) \log \left(1-\frac{\delta p^*}{1-p^*}\right)\right].\label{last}
\end{align}
%\footnotetext{Here we have used that $\log q<0$ and therefore $\log q\limsup_{N\to\infty} a_N=\liminf \log q a_n$}
Observe that $p^*>0$, for otherwise $\sigma b\notin L(X)$, $\occ(\sigma b,s\uh{N})=0$ for all $N$ and the inequality in~\eqref{limsup} would not be obtained. Hence, the multiplying factor on the left is strictly positive. Now if in~\eqref{last} we make the substitution $x=p^{*-1}-1$ we may notice that the function $f(\delta)=(1+\delta)\log(1+\delta)+(x-\delta)\log(1-\delta/x)$ satisfies $f(0)=0$ and $f'(\delta)=\log(1+\delta)-\log(1-\delta/x)>0$ for $0<\delta<x$. Then there is a $c>0$ such that
$
\limsup_{N\to\infty} \log L(s\uh{N})/N \geq c.
$
and there will be infinitely many $N$'s such that $L(s\uh{N})\geq 2^{cN}$, which implies
$\limsup_{N\to\infty} L(s\uh{N})=\infty.
$

\paragraph{Case II.} {\em For all $\tau\in\Theta_k$ we have $m(\tau)=0$.}

Lemma~\ref{lemma2} implies that there are some $d>0$, $\rho\in\Theta_k$, $\sigma\in\Prec_k(\rho)$ and a strictly increasing sequence of natural numbers $(N_j)_{j\in\NN}$ such that
\begin{equation}\label{subseq1}
\lim_{j\to\infty} \frac{\occ(\rho,s\uh{N_j})}{N_j}=0
%\end{equation}
%and
%\begin{equation}\label{subseq2}
\mbox{\qquad and \qquad }
\limsup_{j\to\infty} \frac{\sum_{\tau\in\Theta_k\setminus\{\rho\}} \occ(\sigma\tau,s\uh{N_j})}{N_j}\geq d.
\end{equation}

Notice that $\langle s:k\rangle$ is actually a sequence in $\Theta_k$ (and not just $\Sigma^k$) since $s\in X$  and all words of length $k$ in $s$ must belong to $L(X)$. Also, as mentioned in \S\ref{sec:symdyn}, $P$ induces an irreducible Markov measure $P^k$ on $\Theta_k^\NN$, so we will first construct a $P^k$-martingale on $\Theta_k^*$. Let $p^*=1-P(\sigma \rho|\sigma)<1$ (since $\sigma\in\Prec_k(\rho)$), $(1-p^*)>p^*\delta$, $c$ be a symbol of $\Theta_k$ and $\delta>0$. We define $M$ as follows:
\begin{eqnarray*}
M(\emptyset)&=&1;\\
  M(\tau_1\dots\tau_{l+1} )&=&\begin{cases}
    (1+\delta) M(\tau_1\dots\tau_l) & \text{if $\tau_l=\sigma$ and $\tau_{l+1}\neq\rho$};\\
    \left(1-\frac{\delta p^*}{1-p^*}\right)M(\tau_1\dots\tau_l) & \text{if $\tau_l=\sigma$ and $\tau_{l+1}= \rho$};\\
    M(\tau_1\dots\tau_l) & \text{otherwise.}
  \end{cases}
\end{eqnarray*}
It is easy to check that $M$ is a $P^k$-martingale generated by a DFA.

Observe that
$$
M(\tau_1\dots\tau_l)=\left(1-\frac{\delta p^*}{1-p^*}\right)^{\occ(\sigma \rho,\tau_1\dots\tau_l)}\prod_{\substack{c\in\Theta_k\\ c\neq \rho}}\left(1+\delta\right)^{\occ(\sigma c,\tau_1\dots\tau_l)}.
$$
Write $q=1-\delta p^*/(1-p^*)$ and $r=1+\delta$ and fix $\epsilon>0$ such that $\epsilon<d\log r$. Then~\eqref{subseq1} implies there are infinitely many $N$'s such that
$$\log r\sum_{\substack{\tau\in\Theta_k\\ \tau\neq \rho}}\frac{\occ(\sigma\tau,\langle s:k\rangle\uh{N})}{N}+\log q\frac{\occ(\sigma\rho,\langle s:k\rangle\uh{N})}{N}\geq d\log r-\epsilon.$$
Thus, for $K=d\log r-\epsilon>0$ we have
$
\log M(\langle s:k\rangle\uh{N})/N\geq K>0
$
for infinitely many $N$'s. This implies the martingale succeeds on $\langle s:k\rangle$.

From this martingale $M$ on $\Theta_k^*$ one uses the definition of a $P$-martingale to extend $M$ to a $P$-martingale $\widehat{M}$ on $\Sigma^*$ (it is a routine exercise to check $\widehat{M}$ is well defined as a $P$-martingale and that it is also generated by a DFA), and the fact that $M$ succeeds on $\langle s:k\rangle$ implies that $\widehat{M}$ succeeds on $s$.
\end{proof}

\subsubsection{An extension to sofic shifts}\label{sec:extension-to-sofic}
We now extend Theorem~\ref{maintheorem} to a more general class of measures on sofic subshifts. In order to do so, we need some of the standard results and definitions regarding sofic subshifts.

%\begin{definition}
A \textit{labelled directed graph} on alphabet $\Sigma$ is a tuple $(G,\mathcal{L})$ where $G$ is a directed graph with finite nodes $\mathcal{N}(G)$ and finite edges $\mathcal{E}(G)$ and $\mathcal{L}$ is a function assigning to each edge $e$ in $\mathcal{E}(G)$ a symbol $\mathcal{L}(e)\in\Sigma$.
%\end{definition}

%\begin{definition}
Given a labelled directed graph $(G,\mathcal{L})$, a \textit{path on} $(G,\mathcal{L})$ through states $i_0,\dots,i_l\in\mathcal{N}(G)$ is a finite sequence of symbols $a_1\dots a_l\in\Sigma^*$ for which there are edges $e_1,\dots,e_l\in\mathcal{E}(G)$ such that, for all $1\leq j\leq l$, $i_j$ is the destination node of $e_j$, $i_{j-1}$ is the origin node of $e_j$ and $\mathcal{L}(e_j)=a_j$. The set of paths on $(G,\mathcal{L})$ is denoted $\mathcal{P}_G$.
%\end{definition}

Notice that any labelled directed graph is equivalent to an automaton $M_G$ on $\Sigma$ without an initial state and with a single absorbing non-accepting state. The accepting states of $M_G$ are given by the nodes of $G$, and a transition $\delta(i,a)=j$ whenever there is an edge between $i$ and $j$ labelled $a$. The following definition is then equivalent to this automaton being deterministic.

%\begin{definition}
A labelled directed graph $G$ on alphabet $\Sigma$ is called \textit{right-resolving} if for any symbol $a\in\Sigma$, and any node $i\in\mathcal{N}(G)$ there is at most one edge $e\in\mathcal{E}(G)$ such that $e$ has $i$ as its origin node and $\mathcal{L}(e)=a$.
%\end{definition}

Labelled graphs may be used to represent sofic subshifts in the following way:
\begin{definition}
A \textit{labelled graph presentation} of a sofic subshift $(X,T)$ on alphabet $\Sigma$ is a labelled directed graph $(G,\mathcal{L})$ such that
$L(X)=\mathcal{P}_G$.
\end{definition}
Notice that, given a directed graph $G$ there is a natural 1-step Markov shift $(X_G,T)$ consisting of the admissible sequences of edges. Furthermore, when a labelling function $\mathcal{L}$ is defined on the edges of $G$, the one-block code that maps $e$ to $\mathcal{L}(e)$ induces a factor map $\mathcal{L}^*$ from the Markov shift $(X_G,T)$ to the sofic subshift represented by $(G,\mathcal{L})$.

%\begin{definition}
A labelled graph presentation $(G,\mathcal{L})$ of a sofic subshift $(X,T)$ on alphabet $\Sigma$ is called \textit{minimal} if there is no other presentation $(G',\mathcal{L}')$ of $X$ with strictly fewer nodes,
%\end{definition}
%\begin{definition}
%A labelled graph presentation $(G,\mathcal{L})$ of a sofic subshift $X$ on alphabet $\Sigma$
and it is called \textit{irreducible} if the underlying directed graph is strongly connected
%\end{definition}
%
%We will also need the following definition.
%\begin{definition}
We say that a sofic subshift $(X,T)$ is  \textit{irreducible} if for any words $\sigma,\tau\in L(X)$ there is a word $\rho$ such that $\sigma\rho\tau\in L(X)$.
%\end{definition}
%The following is \cite[Theorem 3.3.2]{Lind95}.
\begin{theorem}(\cite[Theorem 3.3.2]{Lind95})\label{teolind} Any irreducible sofic subshift has a unique (up to graph isomorphism) minimal, irreducible, right-resolving graph presentation.
\end{theorem}
%\begin{definition}
Given a labelled directed graph $(G,\mathcal{L})$ on alphabet $\Sigma$, a \textit{synchronizing word} for $G$ is a word $\alpha=a_1\dots a_l\in\Sigma^*$ for which the set
$$\{i_l\in\mathcal{N}(G)\colon \ \text{there exist }i_0,\dots,i_{l-1}\in\mathcal{N}(G)\text{ such that }\alpha\text{ is a path through }i_0,\dots,i_l\}$$
has a single node. We call that node the \textit{synchronizing node} of $\alpha$.
%\end{definition}
%
That is, when regarding the graph as the equivalent automaton $M_G$, a synchronizing word is one that leaves the automaton in one and only one state after being read, regardless of the state on which its reading began.
Finally, we have \cite[Proposition 3.3.9 and Proposition 3.3.16]{Lind95}:
\begin{theorem}\label{teosincro}A minimal, right-resolving labelled graph presentation of a sofic subshift $(X,T)$ has a synchronizing word.
\end{theorem}
We are now ready to prove the extension of our previous result.
\begin{theorem}\label{teosofic}Let $X$ be an irreducible sofic subshift on alphabet $\Sigma$ and let $(G,\mathcal{L})$ be its minimal, irreducible, right-resolving presentation. Let $(X_G,T)$ be the Markov shift of edge sequences associated to $G$ and let $P$ be an irreducible, invariant, $L(X_G)$-supported, 1-step Markov measure on $\mathcal{E}(G)^\NN$. Let $\mathcal{L}^*\colon X_G\to X$ be the natural factor map induced by $\mathcal{L}$ and let $\nu=P\circ \mathcal{L}^{*-1}$ be the pushforward measure on $\Sigma^\NN$. Let $s\in X$ be not $\nu$-distributed.

\begin{enumerate}
\item\label{teosofic-1} if a synchronizing word appears as a factor of $s$ then there is a $\nu$-martingale generated by a DFA which succeeds on $s$.
Moreover, the only betting factors of this martingale are 1, $(1+\delta)$ and $(1-\delta p^*/(1-p^*))$, where $\delta$ is rational and $p^*=P(\tau\rho\mid\tau)$ or $p^*=1-P(\tau\rho\mid\tau)$ for some $\tau,\rho\in\Sigma^*$.

\item\label{teosofic-2} if no synchronizing word appears as a factor of $s$ then there is a $\nu$-supermartingale generated by a DFA which succeeds on $s$.
Moreover, the only betting factors of this martingale are 1, $(1-\delta^*)$ and $(1+\delta)$, where $\delta^*$ and $\delta$ are rational.
\end{enumerate}
\end{theorem}

%\begin{theorem}\label{teosofic2}Let $X$, $\Sigma$, $(G,\mathcal{L})$, $X_G$, $P$ and $\nu$ be as in Theorem~\ref{teosofic}. If $s\in X$ is not $\nu$-distributed and no synchronizing word appears as a factor of $s$ then there is a $\nu$-supermartingale generated by a DFA which succeeds on $s$.
%
%Moreover, the only betting factors of this martingale are 1, $(1-\delta^*)$ and $(1+\delta)$, where $\delta^*$ and $\delta$ are rational.
%\end{theorem}

\begin{proof}[Proof of item~\ref{teosofic-1} of Theorem~\ref{teosofic}]
Since any leftward extension in $L(X)$ of a synchronizing word is a synchronizing word, we may assume that some prefix of $s$, say $\rho=s\uh{N'}$ is a synchronizing word with synchronizing node $i_0$. Denote by $f(i,a)$ the unique edge whose origin node is $i$ and whose label is $a$, whenever it exists (uniqueness is guaranteed because the presentation is right-resolvable). We construct by induction a sequence $z$ in $\mathcal{E}(G)^\NN$ that records the sequence of edges followed by $s$ after it reaches the synchronizing node:
\begin{itemize}
\item Define $z_1=f(i_0,s_{N'+1})$ and let $i_1$ be the destination node of $z_1$.
\item Assume $z_n$ and $i_n$ are defined. Define $z_{n+1}=f(i_n,s_{N'+n})$ and let $i_{n+1}$ be the destination node of $z_{n+1}$ (notice that $f(i_n,s_{N'+n})$ must exist because $s\in X$ and $(G,\mathcal{L})$ is a presentation of $X$)
\end{itemize}
Notice that by construction $\mathcal{L}^*(z)=T^{N'}(s)$. We will use this fact to show by contradiction that $z$ is not $P$-distributed, and then use the martingale on a DFA that succeeds on $z$ (guaranteed by Theorem~\ref{maintheorem}) to build an appropriate martingale on a DFA that succeeds on~$s$.

For any word $a_1\dots a_n\in\Sigma^*$ let $A(a_1\dots a_n)$ be the set of words $e_1\dots e_n\in\mathcal{E}(G)^*$ such that $\mathcal{L}(e_i)=a_i$ for $1\leq i\leq n$ (equivalently, $A(\alpha)=\mathcal{L}^{*-1}(\alpha)$). Clearly, we have
$
\nu(\alpha)=\sum_{\sigma\in A(\alpha)} P(\sigma)
$
and, by construction of $z$,
$
\occ(\alpha,T^{N'}(s)\uh{N})=\sum_{\sigma\in A(\alpha)} \occ(\sigma,z\uh{N}).
$
Then, if $z$ is $P$-distributed, $T^{N'}(s)$ must be $\nu$-distributed and $T^{N'}(s)$ is $\nu$-distributed if and only if $s$ is so. Hence, $z$ cannot be $P$-distributed. By Theorem~\ref{maintheorem} there is a martingale $\widehat{L}$ generated by a DFA $\widehat{M}=\langle \widehat{Q},\mathcal{E}(G),\hat{\delta},\hat{q}_0,\hat{Q}_f\rangle$ and a function $\hat{b}\colon \hat{Q}\times\mathcal{E}(G)\to\RR$, and such that $\limsup_N \hat{L}(s\uh{N})=\infty$.

Given an edge $e$, we write $d(e)$ for its destination node and $o(e)$ for its origin node. Let $M'$ be a DFA on $\Sigma$ having $Q'=(\hat{Q}\times\mathcal{N}(G))\cup\{q_g\}$ (for some unused garbage state $q_g$) as its set of states, $Q_f'=\hat{Q}_f\times\mathcal{N}(G)$ as its set of final states, $q_0'=(\hat{q}_0,i_0)$ as its initial state and a transition function $\delta'$ defined by:
\begin{align*}
\delta'((q,i),a)&=\begin{cases}
(\hat{\delta}(q,f(i,a)),d(f(i,a))) & \text{if $f(i,a)$ exists;}\\
q_g & \text{otherwise.}
\end{cases}\\
\delta'(q_g,a)&=q_g\quad\text{for all $a$.}
\end{align*}
We also define $b'\colon Q'\times\Sigma\to\RR$ as
\begin{align*}
b'((q,i),a)&=\hat{b}(q,f(i,a))\\
b'(q_g,a)&=1
\end{align*}
Notice that, by construction, the function $f$ defined by $f(\lambda)=1$ and
$
f(\alpha a)=b'(\delta'^*(q_0',\alpha),a)f(\alpha)
$
satisfies
\begin{equation}\label{martexito}\limsup_N f(T^{N'}(s)\uh{N})=\infty.
\end{equation}
since the sequence of betting factors induced by $T^{N'}(s)$ for $M'$ and $b'$ is the same as that induced by $z$ for $\hat{M}$ and $\hat{b}$.

Finally, write $\rho=a_1\dots a_l$ and define the DFA $M=\langle Q,\Sigma,\delta,a_1,Q_f'\rangle$, where
$Q=Q'\cup\{a_1,\dots,a_l\}$
and
$$\delta(q,a)=
\begin{cases}
\delta'(q,a)\quad &\mbox{if $q\in Q'$;}\\
a_2 &\mbox{if $a=q=a_1$;}\\
a_{i+1} &\mbox{if $a=q=a_i$, for $2\leq i\leq l-1$;}\\
q_0' &\mbox{if $a=q=a_{l}$;}\\
a_1 & \mbox{otherwise.}
\end{cases}
$$
This automaton waits till the synchronizing word $\rho$ is read. Once it finishes reading it the automaton transitions to the automaton $M'$ and stays there (see Figure~\ref{fig:automatonM}).
\begin{figure}
\begin{tikzpicture}[->,>=stealth',shorten >=1pt,auto,node distance=2.8cm,
                    semithick, node distance=2cm]
  \tikzstyle{every state}=[minimum size=9mm]

  \node[initial,state]    (A)                    {$a_1$};
  \node[state]            (B) [right of=A]       {$a_2$};
  \node[state]            (C) [right of=B]       {$a_3$};
  \node[state, draw=none] (D) [right of=C]       {$\dots$};
  \node[state]            (E) [right of=D]       {$a_l$};
  \node[state]            (F) [right of=E]       {$q'_0$};

\draw (9,-1.5) rectangle (13,1.5);
\node at (11.7,1.2) {automaton $M'$};

  \path (A) edge [above,bend left] node {$a_1$} (B)
        (B) edge [above,bend left=20] node {$\not=a_1$} (A)
        (B) edge [above,bend left] node {$a_2$} (C)
        (C) edge [above right,bend left=35] node {$\not=a_2$} (A)
        (C) edge [above,bend left] node {$a_3$} (D)
        (D) edge [above,bend left] node {$a_{l-1}$} (E)
        (E) edge [above,bend left=50] node {$\not=a_l$} (A)
        (E) edge [above left,bend left] node {$a_l$} (F);
\end{tikzpicture}
\caption{The automaton $M$. $\rho=a_1\dots a_l$ is a synchronizing word, and $M'$ ``translates'' $\widehat M$, which has inputs on the language of edges, $\mathcal{E}(G)^*$ to the language $L(X)\subseteq\Sigma^*$.}\label{fig:automatonM}
\end{figure}
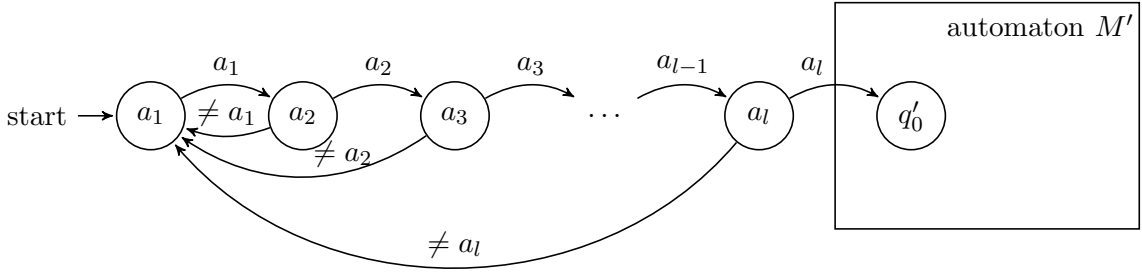
The function $b\colon Q\times\Sigma\to \RR$, computing the betting factors, must then be
$$
b(q,a)=
\begin{cases}1\quad &\text{if $q\notin Q'$;}\\
b'(q,a)&\text{otherwise.}
\end{cases}
$$
From~\eqref{martexito} and the construction of $f$, $M$ and $b$, it follows that the function $L$ defined by
\begin{align*}
L(\lambda)&=1\\
L(\alpha a)&=b(\delta^*(a_1,\alpha),a)L(\alpha)\quad\text{when $\alpha a\in L(X)$}
\end{align*}
satisfies
$
\limsup_N L(s\uh{N})=\infty.
$
Also, by construction, $L$ has the same betting factors as $\hat{L}$, which are 1, $(1+\delta)$ and $(1-\delta p^*/(1-p^*))$.

It remains to show that $L$ is a $\nu$-martingale, i.e.\ that
$
L(\alpha)=\sum_{a\in\Sigma}\nu(\alpha a \mid \alpha)L(\alpha a).
$
By definition of $L$ and $b$, this condition is trivially satisfied when $\rho$ is not a prefix of $\alpha$, since in that case $L(\alpha a)=L(\alpha)=1$. Hence, we only need to show
\begin{equation}\label{condmart}
1=\sum_{a\in\Sigma}\nu(\alpha a\mid\alpha)b(\delta^*(a_1,\alpha),a)
\end{equation}
when $\rho$ is a prefix of $\alpha$, say, $\rho=\alpha\uh{N_0}$.

Observe that
$$\nu(\alpha a\mid\alpha)=\frac{\sum_{\sigma\in A(\alpha)}P(\sigma f(d(\sigma_{|\sigma|}),a))}{\sum_{\sigma\in A(\alpha)}P(\sigma)}$$
and we may define $e=\sigma_{|\sigma|}$ and $h=f(d(e),a)$, independently of $\sigma$, since $\sigma$ is the path of edges followed by $\alpha$ and $\alpha$ has a prefix that is a synchronizing word, so that $\sigma_n$ is the same for all $\sigma\in A(\alpha)$ when $n\geq N_0$. Thus, the fact that $P$ is 1-step Markov implies
$$
\nu(\alpha a\mid\alpha)=\frac{\sum_{\sigma\in A(\alpha)}P(\sigma h\mid \sigma)P(\sigma)}{\sum_{\sigma\in A(\alpha)}P(\sigma)}=P(eh\mid e)\frac{\sum_{\sigma\in A(\alpha)}P(\sigma)}{\sum_{\sigma\in A(\alpha)}P(\sigma)}=P(eh\mid e)=P(\sigma h\mid\sigma),
$$
and writing $\eta=\sigma\dh{|\alpha|-N_0}$ (which is the same for all $\sigma\in A(\alpha)$ by the preceding remark) and $\tau=\mathcal{L}^*(\eta)$ we have that~\eqref{condmart} boils down  to
\begin{align*}%\label{condmart2}
1&=\sum_{h:o(h)=d(e)} P(\sigma h\mid \sigma)b(\delta^*(a_1,\alpha),a)\nonumber\\
&=\sum_{h:o(h)=d(e)} P(\eta h\mid \eta)b'(\delta'^*(q_0',\tau),a)\nonumber\\
&=\sum_{h:o(h)=d(e)} P(\eta h\mid \eta)\hat{b}(\hat{\delta}^*(\hat{q}_0,\eta),f(d(e),a))\nonumber\\
&=\sum_{h:o(h)=d(e)} P(\eta h\mid \eta)\hat{b}(\hat{\delta}^*(\hat{q}_0,\eta),h)\frac{\hat{L}(\eta)}{\hat{L}(\eta)}\nonumber\\
&=\sum_{h:o(h)=d(e)} P(\eta h\mid \eta)\frac{\hat{L}(\eta h)}{\hat{L}(\eta)}.
\end{align*}
This last condition is met because $\hat{L}$ is a $P$-martingale.
\end{proof}

When no synchronizing word appears in $s$, the conditional probabilities that appear in the martingale condition may have infinitely many possible values, so that they will not be computed by a martingale generated by a DFA. Yet, we can still find a supermartingale on a DFA to handle this case.

\begin{proof}[Proof of item~\ref{teosofic-2} Theorem~\ref{teosofic}]
As before, define  $A(\alpha)=\mathcal{L}^{*-1}(\alpha)$.
By Theorem~\ref{teosincro} $(G,\mathcal{L})$ has a synchronizing word $\alpha$, and $\alpha\in X$. Hence, $A(\alpha)$ is not empty, and since $P$ is $L(X_G)$-supported we have $P(\sigma)>0$ for any $\sigma\in A(\alpha)$. Therefore, $\nu(\alpha)=\sum_{\sigma\in A(\alpha)}P(\sigma)>0$.

By hypothesis, $\alpha$ is not a factor of $s$, hence $\occ(\alpha,s\uh{N})=0$ for all $N$. Let $N_\alpha=\max\{N\colon \alpha\uh{N}\text{ appears infinitely many times in $s$}\}\cup \{0\}$ and write $\eta b c=\alpha\uh{N_\alpha+1}$, when $N_\alpha>0$ and $c=\alpha\uh{1}$ when $N_\alpha=0$.

Take any rational $0<\delta^*<1$. For any word $\gamma$ we have that
\begin{equation}\label{sumagama}
\nu(\gamma\eta b c\mid\gamma\eta b)=\sum_{ef\in A(bc)}\left(\sum_{\sigma\in A(\gamma\eta )}P(\sigma e \mid A(\gamma\eta b))\right)P(ef\mid e)
\end{equation}
when $N_\alpha>0$, and
$$
\nu(\gamma c\mid\gamma )=\sum_{f\in A(c)}\left(\sum_{\sigma\in A(\gamma )}P(\sigma \mid A(\gamma ))\right)P(\sigma f\mid\sigma)$$
when $N_\alpha=0$.

Since there are finitely many $ef\in A(bc)$ we may set, in case $N_\alpha>0$,
$$
K=\min\{P(ef\mid e)\colon  P(ef\mid e)\neq 0, ef\in A(bc)\},
$$
and then, noticing (\ref{sumagama}) consists of nonnegative summands and choosing any $f$ for which $P(ef\mid e)\neq 0$
$$
\nu(\gamma\eta b c\mid\gamma\eta b)\geq K\sum_{\sigma e\in A(\gamma\eta b)}P(\sigma e \mid A(\gamma\eta b))=K.
$$

Taking any strictly positive rational $\delta\leq \delta^* K$ we get
\begin{equation}\label{desigsuper}\delta\leq\delta^*K\leq\frac{\delta^*\nu(\rho c\mid\rho )}{1-\nu(\rho  c\mid \rho )}
\end{equation}
for any $\rho=\gamma\eta b$.

In case $N_\alpha=0$ we set $K=\min\{P(\sigma f\mid \sigma)\colon \ P(\sigma f\mid\sigma)\neq 0; f\in A(c)\}$ (which exists because $P$ is Markovian) and ~\eqref{desigsuper} holds.
Then, the function defined by
\begin{eqnarray*}
L(\emptyset)&=&1;\\
  L(\rho a)&=&\begin{cases}
    (1+\delta) L(\rho) & \text{if $\rho\dh{ N_\alpha}=\alpha\uh{N_\alpha}$ and $a\neq c$};\\
    \left(1-\delta^*\right)L(\rho) & \text{if $\rho\dh{ N_\alpha}=\alpha\uh{N_\alpha}$ and $a= c$};\\
    L(\rho) & \text{otherwise.}
  \end{cases}
\end{eqnarray*}
satisfies the supermartingale inequality, since
$$1\geq(1+\delta)(1-\nu(\rho c\mid \rho))+\nu(\rho c\mid\rho)(1-\delta^*)$$
follows from~\eqref{desigsuper} for all $\rho=\gamma\eta b$ (in case $N_\alpha>0$) and for all $\rho$ in case $N_\alpha=0$.

Observe that
$$
L(\rho)=\left(1-\delta^*\right)^{\occ((\alpha\uh{N_\alpha})c,\rho)}\prod_{\substack{a\in\Sigma\\ a\neq c}}\left(1+\delta\right)^{\occ((\alpha\uh{N_\alpha}) a,\rho)}
$$
Since the word $(\alpha\uh{N_\alpha})c$ occurs finitely many times in $s$, while $\alpha\uh{N_\alpha+1}$ occurs infinitely many times in $s$, we conclude that the above function goes to infinity when evaluated on increasing prefixes of $s$.
\end{proof}

\subsection{$P$-martingales on a DFA cannot beat $P$-distributed sequences}\label{sec:mart-cannot-beat}

Our next goal is to prove a converse to Theorem~\ref{maintheorem}, thus providing a complete characterization of sequences that are ``normal'' relative to some irreducible Markov measure, a characterization that generalizes the main result of Schnorr and Stimm in~\cite{Sch72}. Our proof will mirror their ideas closely. The main intuition is that a sequence where the average occurrences of blocks converge to some measure on those blocks should also have the average number of visits to any state of a DFA converge to some measure on the states. The main differences are that, in our case, some states are not final and that the probability of symbols and states are not independent.

To ease notation, we will only consider $1
$-step Markov measures. The reader may check that there is no loss in generality in this, since, as mentioned in \S\ref{sec:symdyn}, any $k$-step irreducible Markov measure on $\Sigma^\NN$ induces a $1$-step irreducible Markov measure $P^k$ on $\Theta_k^\NN$ and any $P$-martingale $M$ generated by a DFA on alphabet $\Sigma$ that succeeds on a sequence $s$ can be regarded as a $P^k$-martingale $M'$ generated by a DFA on alphabet $\Theta_k$. This martingale may not succeed on $\langle s:k\rangle$ but it must succeed on $\langle T^i(s):k\rangle$ for some $i\leq k$. Since $P$-distribution is unaffected by the removal of finitely many symbols, this suffices.

The main result of this section is an analogue of part a) of Theorem 4.1 in~\cite{Sch72}:

\begin{theorem}\label{thm:P-dist-implies-no-success}
Let $L\subseteq\Sigma^*$ be a prolongable and factorial language, let $P$ be an $L$-supported irreducible 1-step Markov measure on $\Sigma^\NN$ and let $s\in X_L$ be $P$-distributed.
Then no $P$-martingale generated by a DFA succeeds on $s$.
\end{theorem}

Take some $M=\langle Q,\Sigma,\delta,q_0,Q_f\rangle$ accepting $L$. We may assume all accepting states in $Q_f$ are reachable from the initial state $q_0$.

\begin{definition}Let $M$ be a DFA and $q\in Q_f$, then $M_q$ is the DFA that has the same states, accepting states, alphabet and transition function as $M$ but which has $q$ as its initial state.
\end{definition}

Notice first that, since our language $L$ is factorial, if a word $\sigma$ is such that there are states $q,q'\in Q_f$ and $\delta^*(\sigma,q)=q'$, then $\sigma\in L$. That is, words that transition between accepting states must belong to the language. Also, factoriality implies that transition to an accepting state is not possible once a state outside $Q_f$ is reached (hence, $q_0\in Q_f$). Thus, we may assume that the complement of $Q_f$ consists of a single state $\tilde{q}$. Similarly, the fact that the language is prolongable implies that from any accepting state there is always a transition to an accepting state.

As in~\cite{Sch72}, we will define the relation $q\rightarrow q'$ when there is a word $\sigma\in L(X)$ such that $\delta^*(\sigma,q)=q'$ and $q\leftrightarrow q'$ if both $q\rightarrow q'$ and $q'\rightarrow q$. From the remarks in the preceding paragraph it is easy to see that $\leftrightarrow$ is an equivalence relation and that it allocates accepting states $q$ to classes $[q]$ different from $[\tilde{q}]$. Also, it is easy to check that the relation $\rightarrow$ induces a relation $\geq$ in the equivalence classes of $Q/\!\!\leftrightarrow$, where $[q]\geq[q']$ if and only if $q\rightarrow q'$ (we write $[q]>[q']$ when this holds and $[q]\neq [q'])$. We will call a class $[q]$ {\em ergodic} if $[q]\neq[\tilde{q}]$ and if there is no $q'\in Q_f$ such that $[q]>[q']$ (i.e., $[q]$ is minimal among the classes of accepting states).

In order to prove the main result of this section we must make some considerations regarding the interaction of Markov measures that are $L$-supported and a DFA that accepts~$L$.

For any $q'\in Q_f$ consider the maps $\phi_{q'}\colon X_L\rightarrow Q^\NN$ and $\Phi_{q'}\colon X_L\rightarrow (\Sigma\times Q)^\NN$ defined by the following rules:
\begin{align*}
\phi_{q'}(x)_1&=q'\\
\phi_{q'}(x)_{n+1}&=\delta(x_n,\phi_{q'}(x)_{n})\\
\Phi_{q'}(x)_n&=(x_n,\phi_{q'}(x)_n)
\end{align*}
The invariant, $L$-supported, irreducible 1-step Markov measure $P$ of Theorem~\ref{thm:P-dist-implies-no-success} allows us to define some random processes, that is, random variables indexed by natural numbers $n$. For any $n\geq 1$ let $W_n$, $Y_n^{q}$ and $Z_n^{q'}$ be measurable functions (i.e. random variables) on the probability space $(X_L,\mathcal{B},P)$ %
\begin{align*}
W_n(x)&=x_n\\
Y_n^{q'}(x)&=\phi_{q'}(x)_n\\
Z_n^{q'}(x)&=\Phi_{q'}(x)_n=(W_n,Y_n^{q})
\end{align*}
As is customary for random variables, we will omit the specification of the element $x$ of the probability space on which the random variable is being evaluated. Also, when $q'=q_0$, we will drop the subscripts and superscripts and write $\phi$, $\Phi$, $Y_n$ and $Z_n$.

It is a commonplace observation in the theory of Markov processes \cite[Theorem 1.1.2]{Norr97} that the random process $(Z_n^{q'})_{n\in\NN}$ is Markov of order 1 or, equivalently, that the measure $P\circ\Phi_{q'}^{-1}$ on $(\Sigma\times Q)^\NN$ is 1-step Markov.
%\footnote{The first chapter of~\cite{Norr97} is a good introduction to the theory of Markov processes and chains.}
Moreover, while the overall measure depends on the choice of $q'$, the transition matrix does not.

For any ergodic class $[q^*]$, write
$${\sf A}_{q^*}=\{(a,q)\in\Sigma\times[q^*]\colon \ \delta(a,q)\notin[\tilde{q}]\}$$
for the tuples of ``admisible" pairs in $\Sigma\times[q^*]$. And for any $(a,q)\in {\sf A}_{q^*}$ we define a measure $\widehat{P}_{a,q}$ on $(\Sigma\times Q)^\NN$ by letting
$$
\widehat{P}_{a,q}(z_1\dots z_l)=P\left(\{x\colon \ \forall\ 1\leq i\leq l,\ Z_{N+i}=z_i\}\mid\{x\colon \ Z_N=(a,q)\}\right)
$$
for any $N$ such that $P(\{x\colon \ Z_N=(a,q)\})>0$ (equivalently, any $N$ for which there is a word $\sigma$ such that $|\sigma|=N-1$ and $\delta^*(\sigma,q_0)=q$). This measure is the probability distribution of the process $(Z_{N+i})_{i\in\NN}$ conditioned on $Z_n=(a,q)$ and by \cite[Theorem 1.1.2]{Norr97} it is a Markov measure independent of $N$ and having the same transition matrix as $P\circ\Phi_{q'}^{-1}$ regardless of $(a,q)$. When restricted to a given ergodic class $[q^*]$ we will denote this transition matrix by
$$
\widehat{P}^{q^*}=(\widehat{P}^{q^*}_{z,z'})_{z,z'\in{\sf A}_{q^*}}.
$$
\begin{observation}\label{obsetransit}
For $z=(a,q)$ and $z'=(a',q')$ we have
$$
\widehat{P}_{z,z'}=
\begin{cases}P(aa'\mid a) & \text{if $\delta(a,q)=q'$;}\\
0& \text{otherwise.}
\end{cases}
$$
\end{observation}

%Moreover,
%\begin{align}
%\sum_{\substack{\sigma\in(\Sigma\times Q)^{n-1}\\ \widehat{P}_{a,q}(\sigma z_n)>0}}\widehat{P}_{a,q}(\sigma z_n\dots z_{n+l}\mid \sigma z_n)&=P(\{x: \forall\ n\leq i\leq n+l,\ Z_{M+i}=z_{i}\}\mid\{x: Z_{M+n}=z_n\})=\nonumber\\
%&=P(\{x: \forall\ 0\leq i\leq l,\ Z^{q'}_{N+i}=z_{m+i}\}\mid\{x: Z^{q'}_N=z_n\})\nonumber
%\end{align}
%for any $q'$, $N$, $M$, $n$ and $l$ such that the conditional probabilities on the right hand side are well defined.

\begin{lemma} Let $[q^*]$ be an ergodic class and $(a,q)\in{\sf A}_{q^*}$. Then the measure $\widehat{P}_{a,q}$ is supported on ${\sf A}_{q^*}^\NN$ and has an irreducible transition matrix
\end{lemma}

\begin{proof}Suppose that, for some $N_1$, $\widehat{P}_{a,q}(\{z\in(\Sigma\times Q)^\NN\colon \ z_{N_1}=(a',q')\})>0$. This means there are words $\tau,\rho\in L$ such that $\delta^*(\tau,q_0)=q$, $|\rho|=N_1-1$, $\delta^*(\rho,\delta(a,q))=q'$ and $P(\tau a\rho a')>0$. But since $P$ is $L(X)$-supported this means $\tau a\rho a'\in L(X)$ and therefore $\hat{q}=\delta^*(\tau a\rho a',q_0)$ is an accepting state and $\hat{q}=\delta(\rho a',q)$, so that $[\hat{q}]\leq [q]$. But since $[q]$ is an ergodic class it must be the case that $[q']=[q]=[q^*]$. Hence, $\widehat{P}_{a,q}$ is supported on ${\sf A}_{q^*}^\NN$.

To see that the transition matrix of $\widehat{P}_{a,q}$ is irreducible, take any $z_1=(a_1,q_1),z'=(a',q')\in{\sf A}_{q^*}$. We want to find some word $\rho=z_2\dots z_m$ in ${\sf A}_{q^*}^*$ such that
\begin{equation}\label{probirred}
\widehat{P}_{a,q}((a_1,q_1)\rho(a',q'))>0.
\end{equation}

Let $q_2=\delta(a_1,q_1)$. Since $\widehat{P}_{a,q}$ is supported on ${\sf A}_{q^*}^\NN$ it follows that $q_2$ is also in the ergodic class $[q^*]$, and since $q'$ also belongs to the same class $[q^*]$ by hypothesis, then there must be a word $\sigma\in L(X)$ such that $\delta^*(\sigma,q_2)=q'$. Write $\sigma=a_2\dots a_m$ and inductively define $q_{i+1}=\delta(a_i,q_i)$ for $2\leq i< m$. Let us show the word $z_2\dots z_m$ for $z_i=(a_i,q_i)$ satisfies~\eqref{probirred}. Take any word $\tau$ such that $\delta^*(\tau,q_0)=q$. Then $\delta^*(\tau aa_1 \sigma,q_0)=q'$ and since $(a',q')\in{\sf A}_{q^*}^*$ then $\delta(a',q')\notin[\tilde{q}]$, so that $\delta^*(\tau a a_1 \sigma a',q_0)\in Q_f$ and therefore $\tau a a_1 \sigma a'\in L(X)$. Since $P$ is $L(X)$-supported, this means $P(\tau a a_1 \sigma a')>0$, and from $[\tau a a_1 \sigma a']\subseteq \Phi^{-1}\left(T^{-|\tau|}[(a,q)(a_1,q_1)z_2\dots z_m(a',q')]\right)$ we derive~\eqref{probirred}.
\end{proof}

Now, $\widehat{P}_{a,q}$ is 1-step Markov with an irreducible transition matrix. Given a 1-step Markov measure with irreducible transition matrix, the ergodic theorem for Markov processes (Theorem 1.10.2 from~\cite{Norr97}) ensures that the Cesaro averages of cylinder characteristic functions converge almost surely to a constant that depends only on the transition matrix. In our context, that result has to be restated in the following form:
\begin{theorem}\label{lemaergodico}Let $P$ and $L$ be as in Theorem~\ref{thm:P-dist-implies-no-success} and $M$ be a DFA accepting $L$. Let $[q^*]$ be an ergodic class and $(a,q),(a',q')\in{\sf A}_{q^*}$. Then there is some constant $ k_{a',q'}$ independent of $(a,q)$ such that
$$
\widehat{P}_{a,q}\left(z\colon \sum_{i=1}^N\frac{\mathcal{X}_{(a',q')}(z_i)}{N}\to k_{a',q'}\right)=1.
$$
Moreover, the vector $\psi_{q^*}=( k_{a',q'})_{(a',q')\in{\sf A}_{q^*}}$ is a distribution on ${\sf A}_{q^*}$ (that is, it has nonnegative entries that add up to $1$) and is a left eigenvector of the transition matrix $\widehat{P}^{q^*}$, that is
$\psi_{q^*}\widehat{P}^{q^*}=\psi_{q^*}.$

\end{theorem}

At this point, we would like to prove an analogue of Lemma 4.5 from~\cite{Sch72}, which states a precise formulation of the idea that if a sequence $s$ is $P$-distributed then the joint sequence of visited states and symbols should also be distributed according to some measure derived from $P$. If the sequence of visited states were eventually concentrated in some ergodic class $[q^*]$, then $ k_{a,q}$ would be the natural candidate for that derived measure. The following simple but useful result will allow us to make that assumption regarding an eventual ergodic class~$[q^*]$.

\begin{lemma}\label{lemma3}Given a DFA $M$ accepting a factorial and prolongable language $L$ and a class of accepting states $[q]\in Q_f/\!\!\leftrightarrow$, there is a word $\sigma\in L$ such that for all $s\in [q]$, the class $[\delta^*(\sigma,s)]$ is ergodic or is equal to $[\tilde{q}]$.
\end{lemma}
\begin{proof}If $[q]$ is ergodic then any word $\sigma\in L$ will do, so let us assume $[q]$ is not ergodic.
Let us write $[q]=\{q_0,\dots,q_m\}$. The proof will show by induction on $i$that there are words $\sigma_i\in L$ such that, for all $j\leq i$, $[\delta^*(\sigma_i,q_j)]=[\tilde{q}]$ or $[\delta^*(\sigma_i,q_j)]$ is ergodic.
For $i=1$, since $[q]$ is not ergodic there must be some ergodic $[q']$ such that $[q]>[q']$. Hence, we can choose a word $\sigma_1\in L$ such that $\delta^*(\sigma_1,q_1)\in [q']$.

For the inductive step, if $[\delta^*(\sigma_i,q_{i+1})]=[\tilde{q}]$ or $[\delta^*(\sigma_i,q_{i+1})]$ is ergodic then we are done. Otherwise, there must be some ergodic $[q']$ such that $[\delta^*(\sigma_i,q_{i+1})]>[q']$. Hence, there is a word $\sigma$ such that $[\delta^*(\sigma,\delta^*(\sigma_i,q_{i+1}))]=[q']$ is ergodic, and we choose $\sigma_{i+1}=\sigma_i\sigma$, which belongs to $L$ because $L$ is factorial and $\sigma_{i+1}$ transitions between accepting states $q_{i+1}$ and $q'$ (for some representative of $[q']$).
\end{proof}

We need some notation for the 2-tuples of letters and states visited by words of finite and fixed length. For this purpose, let $\Phi_q^k\colon \Sigma^k\rightarrow (\Sigma\times Q)^k$ be defined as
\begin{align*}
\Phi_q^1(a)&=(a,q)\\
\Phi_q^{k+1}(\sigma a)&=\Phi_q^k(\sigma)(a,\delta(\Phi_q^k(\sigma)_k
\end{align*}
We now get the desired generalization of Lemma 4.5 from~\cite{Sch72}.

\begin{lemma}\label{lemalim}
Let $L$, $P$ and $s$ be as in Theorem~\ref{thm:P-dist-implies-no-success} and $M$ be some DFA accepting $L$.
%Let $M$ be an automaton\santi{DFA?} accepting a prolongable and factorial language $L$, let $P$ be a $L$-supported irreducible 1-step Markov measure and let $s\in X_L$ be $P$-distributed.
Then there is some ergodic class $[q^*]$ such that, for all $(a',q')\in{\sf A}_{q^*}$, we have
\begin{equation}\label{limnorm}
\lim_N\sum_{n=1}^N\frac{\mathcal{X}_{(a',q')}(Z_n(s))}{N}= k_{a',q'}.
\end{equation}
\end{lemma}
\begin{proof}Notice that, since $s\in X$, $M$ on input $s\uh{N}$, will never reach the garbage state $\tilde{q}$, since $s\in X_L$ and $L$ is a factorial language. Moreover, we will show that there is some $N_0$ such that, for all $N>N_0$, $M$ stays within some ergodic class $[q^*]$ after reading the first $N$ symbols of $s$, i.e., $[\delta^*(s\uh{N},q_0)]=[q^*]$ for all $N>N_0$.

Observe that if $M>N$, then $[\delta^*(s\uh{M},q_0)]=[\delta^*(s\uh{N},q_0)]$ or $[\delta^*(s\uh{M},q_0)]>[\delta^*(s\uh{N},q_0)]$, and since there are only finitely many classes in $Q_f/\!\!\leftrightarrow$ this means there is some $N_0$ such that, for all $N>N_0$, $M$ stays in the same class, that is,
\begin{equation}\label{lastclass}
[\delta^*(s\uh{N_0},q_0)]=[\delta^*(s\uh{N},q_0)].
\end{equation}
We will call this class $[q^*]$ and claim it is ergodic.

Indeed, by Lemma~\ref{lemma3} we could choose some word $\sigma\in L(X)$ such that, after reading it from any state in $[q^*]$, $M$ either reaches the garbage state $\tilde{q}$ or reaches a state within an ergodic class. Since $s\in L(X)$ and $P$ is $L$-supported, we have $P(\tau)>0$, and since $L$ is $P$-distributed, $\sigma$ occurs (infinitely many times) in $s$. Since $\tilde{q}$ cannot be reached when a subword of $s$ is read as input, it follows that an ergodic class $[q']$ is reached at some $N>N_0$. By equation~\eqref{lastclass} it follows that $[q^*]=[q']$ and $[q^*]$ is ergodic.

Fix some $(a',q')\in{\sf A}_{q^*}$. Let us now apply Theorem~\ref{lemaergodico} to $[q^*]$. We have
\begin{equation}
\widehat{P}_{a,q}\left(z: \lim_N\sum_{i=1}^N\frac{\mathcal{X}_{(a',q')}(z_i)}{N}= k_{a',q'}\right)=1\nonumber
\end{equation}
for all $(a,q)\in{\sf A}_{q^*}$, which in turn implies that for all $\epsilon,\epsilon'>0$ there is $k_0$ such that for all $k\geq k_0$ and $(a,q)\in{\sf A}_{q^*}$,
\begin{equation}\label{limergodico}
\widehat{P}_{a,q}\left((a_1,q_1)\dots(a_l,q_k)\in(\Sigma\times Q)^k:\ \left|\sum_{i=1}^k\frac{\mathcal{X}_{(a',q')}((a_i,q_i))}{N}- k_{a',q'}\right|<\epsilon\right)>1-\epsilon'.
\end{equation}
Notice that~\eqref{limergodico} can be rewritten as
\begin{equation}\label{limergodico2}
P\left(a\sigma\in\Sigma^k:\ \left|\sum_{i=1}^k\frac{\mathcal{X}_{(a',q')}(\Phi_q^k(a\sigma)_i)}{N}- k_{a',q'}\right|<\epsilon\right)/P(a)>1-\epsilon'.
\end{equation}
Write
\begin{align*}
{\sf B}_{q^*}^k(\epsilon,a)&=\left\{a\sigma\in\Sigma^k:(\forall q\in[q^*], (a,q)\in{\sf A}_{q^*}) \left|\sum_{n=1}^k\frac{\mathcal{X}_{(a',q')}(\Phi_q^k(a\sigma)_n)}{N}- k_{a',q'}\right|<\epsilon\right\}
\\
{\sf B}_{q^*}^k(\epsilon)&=\bigcup_{a\in\Sigma}{\sf B}_{q^*}^k(\epsilon,a)
\end{align*}
These are just the words of length $k$ for which the average occurrences of all pairs $(a,q),\in {\sf A}_{q^*}$ are within an $\epsilon$ distance of their limit. Then it is a standard exercise in probability theory to see that~\eqref{limergodico2} and the finiteness of ${\sf A}_{q^*}$ and $[q^*]$ imply that for all $a$ $(a,q)\in{\sf A}_{q^*}$ for some $q$, and for all $\epsilon,\epsilon'>0$ there $k_0$ such that for all $k\geq k_0$ we have
\begin{equation}\label{cota0}
P\left(\bigcup \{ [\sigma] \colon \sigma\in{\sf B}_{q^*}^k(\epsilon,a) \}\right)/P(a)>1-\epsilon'.
\end{equation}

Let us see that this last inequality holds for all $a$, that is, for all $a$ there is some $q$ satisfying $(a,q)\in{\sf A}_{q^*}$. Indeed, the irreducibility of the transition matrix of the $L$-supported measure $P$ implies that for all $a\in\Sigma$ and all $\sigma\in\Sigma^*$ there is some $\rho\in\Sigma^*$ such that $\sigma\rho a\in L$. Take any $\sigma$ such that $\delta^*(\sigma,q_0)\in[q^*]$ and then take some $\rho$ such that $\sigma\rho a\in L$. This implies $\delta^*(\sigma\rho a,q_0)\notin[\tilde{q}]$ and since $[q^*]$ is an ergodic class this also implies that $\delta^*(\sigma\rho,q_0)\in[q^*]$. So $(a,\delta^*(\sigma\rho,q_0))\in{\sf A}_{q^*}$.

Then from~\eqref{cota0} we derive that for all $\epsilon,\epsilon'>0$ there is $k_0$ such that for all $k\geq k_0$ we have
\begin{equation}\label{cota1}
 P\left(\bigcup \{ [\sigma] \colon \sigma\in{\sf B}_{q^*}^k(\epsilon)\}\right)>1-\epsilon'.
\end{equation}
Now, remember $\langle s:k\rangle$ is the sequence $s$ read as a sequence in $(\Sigma^k)^\NN$, then the $P$-distribution of $s$ implies that for all $\epsilon''>0$ and $k\in\NN$ there is $M_0$ such that for all $\sigma\in\Sigma^k$ and $M\geq \max(M_0,2/\epsilon'')$, we have
\begin{equation}\label{cota2}
 \left|\frac{\occ(\sigma,\langle s:k\rangle\uh{M})}{M}-P(\sigma)\right|<\frac{\epsilon''}{2|\Sigma|^k}
\end{equation}
Take $N_1$ such that $\delta^*(s\uh{N_0},q_0)=q^*$ for $q^*$ some representative of $[q^*]$. Since finitely many summands do not alter the convergence of Cesaro limits, we may substitute $s'=T^{N_1}s$ for $s$ and rewrite~\eqref{limnorm} as
\begin{equation*}%\label{limnorm2}
\lim_N\sum_{n=1}^N\frac{\mathcal{X}_{(a',q')}(Z^{q^*}_n(s'))}{N}= k_{a',q'}.
\end{equation*}
Given $\epsilon,\epsilon',\epsilon''>0$ such that $\epsilon+\epsilon'+\epsilon''<\eta$ for some $\eta>0$, take some $k$ satisfying~\eqref{cota1} and take some $M_0$ satisfying~\eqref{cota2} for this $k$.  For any $M\geq \max(M_0,2/\epsilon'')$ we consider $N=kM$ and notice that
$$
\left|\sum_{n=1}^N\frac{\mathcal{X}_{(a',q')}(Z^{q^*}_n(s'))}{N}- k_{a',q'}\right|=\left|\sum_{j=0}^{M-1}M^{-1}\sum_{i=1}^k\frac{\mathcal{X}_{(a',q')}(\Phi^k_{Y_{jk+1}}(\langle s':k\rangle_j)_i)}{k}- k_{a',q'}\right|,
$$
where all $Y_{jk+1}$ are guaranteed to be in $[q^*]$ by our substitution of $s$. This implies (using first~\eqref{cota1} and then~\eqref{cota2}) that
\begin{align*}
\left|\sum_{n=1}^N\frac{\mathcal{X}_{(a',q')}(Z^{q^*}_n(s'))}{N}- k_{a',q'}\right|&<\epsilon\sum_{\sigma\in{\sf B}_{q^*}^k(\epsilon)}\frac{\occ(\sigma,\langle s':k\rangle\uh{M})}{M}+2\sum_{\substack{\sigma\in\Sigma^k\\ \sigma\notin {\sf B}_{q^*}^k(\epsilon)}}\frac{\occ(\sigma,\langle s':k\rangle\uh{M})}{M}\nonumber\\
&<\epsilon+2\Big|\sum_{\substack{\sigma\in\Sigma^k\\ \sigma\notin {\sf B}_{q^*}^k(\epsilon)}} P(\sigma)\Big|+2\sum_{\substack{\sigma\in\Sigma^k\\ \sigma\notin {\sf B}_{q^*}^k(\epsilon)}}\left|\frac{\occ(\sigma,\langle s':k\rangle\uh{M})}{M}-P(\sigma)\right|\nonumber\\
&<\epsilon+1-P\Big(\bigcup_{\sigma\in{\sf B}_{q^*}^k(\epsilon)}[\sigma]\Big)+2|\Sigma|^k\frac{\epsilon''}{2|\Sigma|^k}<\epsilon+\epsilon'+\epsilon''<\eta.
\end{align*}
For $N=kM+l$ (where $1\leq l<k$) we have
\begin{align*}
\left|\sum_{n=1}^N\frac{\mathcal{X}_{(a',q')}(Z^{q^*}_n(s'))}{N}- k_{a',q'}\right|&\leq\left|\sum_{n=1}^{Mk}\frac{\mathcal{X}_{(a',q')}(Z^{q^*}_n(s'))}{Mk}- k_{a',q'}\right|+\frac{l}{N}+|kM/N-1|\nonumber\\
&<\eta+\frac{2}{M}<\eta+\epsilon''
\end{align*}
This completes the proof.
\end{proof}

%Finally, we are ready to prove Theorem~\ref{thm:P-dist-implies-no-success}:

\begin{proof}[Proof of Theorem~\ref{thm:P-dist-implies-no-success}] Since our martingale is generated by a DFA with betting factors function $b$, states set $Q$ and transition function $\delta$, we know it satisfies
\begin{equation}\label{martdfa}f(\sigma)=f(\emptyset)\prod_{\substack{a\in\Sigma\\q\in Q}} b(a,q)^{\occ\left((a,q),\Phi_{q_0}^{|\sigma|}(\sigma)\right)}.
\end{equation}

By the same reasoning used in the first part of the proof of Lemma~\ref{lemalim} we know there is some $N_1$ such that $\delta^*(s\uh{N},q_0)\in [q^*]$ for all $N\geq N_1$ and some ergodic class $[q^*]$, From~\eqref{martdfa} it is clear that success of $f$ on $s$ is equivalent to success of $f'$ on $T^{N_1}(s)$, where $f'$ is the martingale that is generated by a DFA having $Q'=[q^*]\cap[\tilde{q}]$ as its set of states, $q_1=\delta^*(s\uh{N_1},q_0)$ as its initial state (notice $[q_1]=[q^*]$), and the restriction of $\delta$ and $b$ to $Q'$ as its transition and betting factors functions, respectively. That is, we restrict our analysis to the case in which $M$ starts and stays within a single ergodic class.

Since $s'=T^{N_1}(s)\in X_L$ we know $s'$ will only visit tuples $(a,q)\in {\sf A_{q_1}}$, that is, for all $N$, $\Phi_{q_1}^{N}(s'\uh{N})$ is in ${\sf A}_{q_1}^*$. Then from Lemma~\ref{lemalim} and~\eqref{martdfa} we have
$$
\lim_N\frac{(f'(s'\uh{N}))^{1/N}}{\prod_{(a,q)\in {\sf A_{q_1}}}b(a,q)^{ k_{a,q}}}=1.
$$
If $r=\prod_{(a,q)\in {\sf A_{q_1}}}b(a,q)^{ k_{a,q}}<1$, then
$f'(s'\uh{N})<(r+\epsilon')^N$ for $r+\epsilon'<1$ for large enough $N$. Hence $\lim f'(s'\uh{N})=0$ and the martingale does not succeed on $s'$.

Let us now show that for fixed $q\in[q_1]$
$$
U_q=\prod_{\substack{a\in\Sigma\\(a,q)\in {\sf A_{q_1}}}}b(a,q)^{ k_{a,q}}\leq 1.
$$
Indeed, by Observation~\ref{obsetransit}, Theorem~\ref{lemaergodico} and convexity of the logarithm we get
\begin{align}\label{cuentaconvex}
\log U_q&=\sum_{\substack{a\in\Sigma\\(a,q)\in {\sf A_{q_1}}}} k_{a,q}\log b(a,q)\nonumber\\
&=\sum_{(a',q')\in{\sf A_{q_1}}} k_{a',q'}\sum_{\substack{a\in\Sigma\\(a,q)\in {\sf A_{q_1}}}}\widehat{P}^{q_1}_{(a',q'),(a,q)}\log b(a,q)\nonumber\\
&\leq\sum_{(a',q')\in{\sf A_{q_1}}} k_{a',q'}\log \left(\sum_{\substack{a\in\Sigma\\(a,q)\in {\sf A_{q_1}}}}\widehat{P}^{q_1}_{(a',q'),(a,q)}b(a,q)\right)\nonumber\\
&=\sum_{\substack{(a',q')\in{\sf A_{q_1}}\\\delta(a',q')=q}} k_{a',q'}\log \left(\sum_{a\in\Sigma}P(a'a\mid a')b(a,q)\right)\nonumber\\
&=\sum_{\substack{(a',q')\in{\sf A_{q_1}}\\\delta(a',q')=q}} k_{a',q'}\log \left(1\right)=0,
\end{align}
where the last line follows from the fact that $b(a,q)$ are the betting factors of a $P$-martingale.

It follows that $r=\prod_q U_q\leq 1$ and equality is achieved if and only if $U_q=1$ for all $q\in[q_1]$. But, by strict convexity in~\eqref{cuentaconvex}, $\log U_q=0$ if and only if $b(a,q)$ is the same for all~$a$. Then $r=\prod_q U_q= 1$ if and only if the betting factors are constant at each state, implying $f'(\sigma a)=f'(\sigma)$ for all words $\sigma$ and all $a$. Clearly, a constant martingale cannot succeed on any sequence, and the result follows.
\end{proof}

\section{$\beta$-expansions and Pisot numbers}\label{sec:beta-exp-and-pisot}
In this section we introduce some definitions and known results on the representation of reals in non-integer bases, and Pisot numbers. All these material will be needed for our result of~\S\ref{sec:polyrndness}.

\subsection{$\beta$-expansions}\label{sec:betaexp}
Let us now introduce a way of representing real numbers in a non-integer base $\beta$. Most of the presentation and the definitions are taken from~\cite{Bertrand86}. Let $\lfloor x\rfloor$ and $\lceil x\rceil$ be the floor and ceiling of $x$, respectively, and let $\{x\}$ denote the integer and fractional part.

Let $\beta$ be a real number greater than 1. Any real number $x$ has a unique $\beta$-expansion $s^\beta_{0},s^\beta_{1},\dots$ such that
\begin{equation}\label{betasum}
x=s^\beta_{0}+\sum_{n>0}\frac{s^\beta_{n}}{\beta^n},
\end{equation}
where $s^\beta_{n}$ are nonnegative integers, $0\leq s^\beta_{n}<\beta$ for $n>0$ and any of the following equivalent conditions are met:
\begin{enumerate}
\item $\forall n\geq 0 \sum_{i>n}(s^\beta_{i}/\beta^i)<1/\beta^n$
\item $s^\beta_{n}$ is defined inductively in the following way:
\begin{center}
\begin{tabular}{rclrcl}
$s^\beta_{0}$   & $=$ & $\lfloor x\rfloor,$         & $r_0$     & $=$ & $\{x\}$\\
$s^\beta_{n+1}$ & $=$ & $\lfloor\beta r_n\rfloor,$ & $r_{n+1}$ & $=$ & $\{\beta r_n\}$
\end{tabular}
\end{center}
\end{enumerate}
This expansion coincides with the usual definition given for $\beta$ an integer base.
Notice that the $\beta$-expansion of the real number $\beta$ need not be eventually periodic, in particular, it need not be finite, that is, eventually 0 (of course it is when $\beta$ is an integer).
It is easy to check from the definition of a $\beta$-expansion that if $\beta$ had a finite expansion then it satisfies $\beta=a_0+\frac{a_1}{\beta}+\dots+\frac{a_s}{\beta^s}$ and the periodic sequence $a_0 a_1\dots (a_s-1)a_0 a_1\dots (a_s-1)\dots$ would also satisfy~\eqref{betasum} (but not the following two equivalent conditions, since the $\beta$-expansion is unique).

We will refer to such a periodic sequence as the \textit{periodic} $\beta$-expansion of $\beta$ and we we will write it $\hat{\beta}$ (notice that the periodic $\beta$-expansion may only apply to the base $\beta$).
We will also write $\mathfrak{s}(\beta)$ for the $\beta$-expansion of $\beta$ in case it is not terminated by an infinite sequence of 0's and $\mathfrak{s}(\beta)=\hat{\beta}$ otherwise.

\begin{example}The periodic 2-expansion of 2 is $1^\infty$, whereas its 2-expansion is $20^\infty$.
Let $\phi$ be the golden number satisfying $\phi^2=\phi+1$. The periodic $\phi$-expansion of $\phi$ is $101010\dots$, whereas its $\phi$-expansion is $110^\infty$.
\end{example}

Given a base $\beta$, let  $\Sigma_{\beta}=\{0\dots \lceil\beta\rceil-1]\}$ and let $p_\beta\colon [0,1)\rightarrow \Sigma_\beta^\NN$ be the one-to-one mapping that sends each $x\in [0,1)$ to the fractional part of its $\beta$-expansion
$$
s^\beta_{1}s^\beta_{2}\dots s^\beta_{n}\dots
$$
%(\textit{fractional} naturally means that the integer part $s^\beta_{0}$ is excluded).
%
Notice that $p_\beta(1)=\mathfrak{s}(\beta)$ (strictly speaking, $p_\beta$ is defined on $[0,1)$ but it is trivially extended to $[0,1]$ by continuity).

If $\Sigma$ is a finite set of digits, as in the definition of the mapping $p_\beta$, then the natural ordering of those digits induces a lexicographic order $\leqlex$ on the full shift.

\begin{theorem}\cite[p.\ 273]{Bertrand86}\label{betacaract}
If $\beta>1$ is a real base then the image $p_\beta([0,1))$ is the set
$$
\{s\in \Sigma_\beta^\NN \colon (\forall n)\ T^ns\ltlex \mathfrak{s}(\beta)\}.
$$
\end{theorem}

Notice that the closure of the set above, that is,
$$
\{s\in \Sigma_\beta^\NN \colon (\forall n)\ T^ns\leqlex \mathfrak{s}(\beta)\}.
$$
is a subshift of $\Sigma_\beta^\NN$. In fact, there is a nice converse to the above theorem.

\begin{theorem}\cite[p. 274]{Bertrand86} Suppose that for some alphabet $\Sigma=\{0,\dots,k\}$ we have that $(X,T)$ is a subshift %of $\Sigma^\NN$
such that
$$
X=\{s\in \Sigma^\NN \colon (\forall n)\ T^ns\leqlex s^*\}
$$
for some $s^*$ in $\Sigma^\NN$ which satisfies
$
(\forall n)\ T^ns^*\leqlex s^*.
$
Then $X$ is the closure of $p_\beta([0,1))$  for some real base $\beta$.
\end{theorem}
This allows us to define the following:
\begin{definition}
Given some real number $\beta>1$, the {\em $\beta$-shift} is the subshift $(X_\beta,T)$, %of $\Sigma_{\beta}^\NN$ ,
where
$$
X_{\beta}=\{s\in \Sigma_{\beta}^\NN \colon (\forall n\in\NN)\ T^ns\leqlex \mathfrak{s}(\beta)\}.
$$
\end{definition}

\begin{example}\label{ex:phi-shift}
The $2$-shift is the full shift $\{0,1\}^\NN$ (that is, the Cantor set with the shift operator).
The $\phi$-shift %($\phi$ the golden number)
is the set of infinite sequences on $\{0,1\}$ such that no two 1's occur consecutively in them. In fact, this shift is Markov and it is the Markov shift which had the sofic shift of Example~\ref{examplesofic} as a factor.
\end{example}

%If $s\in \Sigma^\ZZ$, we write $\pi$ for the projection to $\Sigma^\NN$ at the origin. That is,
%$$
%\pi(s)=s_1s_{2}\dots\footnotemark
%$$

%\begin{definition}
%Given some real number $\beta>1$, its {\em bilateral $\beta$-shift} is the subshift $(X_\beta,T)$ of $\Sigma_{\beta}^\ZZ$, where
%$$
%X_{\beta}=\{s\in \Sigma_{\beta}^\ZZ : (\forall i\in \ZZ)\  T^i\pi(s) \leqlex \mathfrak{s}(\beta)\}.
%$$
%\end{definition}
%
\begin{theorem}\cite[Theorem 1]{Bertrand86} Let $\beta>1$ be a real base. Then the $\beta$-shift is a Markov shift if and only if the $\beta$-expansion of $\beta$ is finite, and it is sofic if and only if the $\beta$-expansion of $\beta$ is periodic.
\end{theorem}
A result similar to Theorem~\ref{teoparry2} for $\beta$-shifts was also proved by Parry in~\cite{Parry60}:
\begin{theorem}\label{teoparry}Given a real base $\beta>1$, there is a unique probability measure $\widehat{P}_{\beta}$ on $[0,1)$ such that $P_{\beta}=\tilde{P}_{\beta}\circ p_\beta^{-1}$ is an invariant measure for the $\beta$-shift of maximal metric entropy on $X_{\beta}$. Moreover, $\widehat{P}_{\beta}$ has the closed expression%
\begin{equation}\label{eqn:teoparry}
\widehat{P}_\beta([a,b])=\int_a^b\sum_{n=1}^\infty\mathbbm{1}_{[0,T_\beta^n(1))}(x)\frac{1}{\beta^n}dx,
\end{equation}
and if $(X_{\beta},T)$ is a Markov shift with a grammar of wordlength $k-1$, then $P_{\beta}$, called the {\em Parry measure}, is a $k$-step Markov measure.
\end{theorem}
Notice that the above expression implies that there are positive $k$ and $k'$ such that
\begin{equation}\label{eqn:bound_hatP}
k'\lambda(A)\leq \widehat{P}_\beta(A)\leq k \lambda(A)
\end{equation}
for any Borel subset $A$ of $[0,1)$, and $\lambda$ the Lebesgue measure.

\subsection{Pisot numbers} \label{sec:pisot}
While constructive considerations make us think of rational numbers as the closest relatives of integers, the analysis of real base expansions forces us to consider the ``dynamic" properties of real numbers, and from a dynamical viewpoint non-integer rational numbers are quite distinct from integers. The following definition will introduce us to the closest analog of an integer from a dynamic point of view.

\begin{definition}A real number $\beta$ is a {\em Pisot number} if $\beta>1$ and $\beta$ is the root of a monic polynomial in integer coefficients, such that all its conjugate values (that is, all the other roots of its minimal polynomial) have absolute values strictly less than 1.
\end{definition}
This purely algebraic is interesting for our purposes because of the next remarkable property. Let $\|x\|$ denote the distance from $x$ to its closest integer
\begin{theorem}\cite[Lemma 1]{Bertrand86} A real number $\beta>1$ is a Pisot number if and only if
$
\sum_{n\geq0} \|\beta^n\| %\quad \text{converges.}
$ converges.
\end{theorem}

Pisot numbers are then ``asymptotically integer" in a strong sense. Notice that all integers $n>1$ are Pisot numbers, but no non-integer rational number is Pisot, since the only rationals which are roots of monic polynomials in $\ZZ$ are the integers. The following results relate Pisot numbers and $\beta$-expansions.
\begin{theorem}\cite[Theorem 5]{Bertrand86}\label{pisotsofic} If $\beta$ is a Pisot number, $\mathfrak{s}(\beta)$ is eventually periodic and $X_{\beta}$ is a sofic subshift.
\end{theorem}

\begin{theorem}\cite[Corollary 9]{Bertrand86}\label{thm:pisot-ud}
Let $\beta$ be a Pisot number, $x$ be a real number with $\beta$-expansion $s$ %=s_0s_1\dots$
and assume that $s$ is $P_{\beta}$-distributed. Then $(x\beta^n)_{n\geq 0}$ is u.d.\ modulo one.
\end{theorem}

None of the above implications has a true converse.

\section{Polynomial time randomness}\label{sec:polyrndness}

In this section we will use the martingale constructed in \S\ref{sec:main} to show that if $x\in[0,1]$ is a real whose binary expansion is polynomial time random (i.e.\ no feasible martingale succeeds on it), then $(x\beta^n)_{n\in\NN}$ is u.d.\ modulo one for any Pisot $\beta$.

The reasoning follows by contradiction: if $x$ is such that $(x\beta^n)_{n\geq 0}$ is not u.d.\ modulo one for some Pisot base $\beta$, we know by Theorem~\ref{thm:pisot-ud} that there is some word
$\sigma$ in $\Sigma_\beta^*$ whose average occurrences in the $\beta$-expansion of $x$ do not converge to $P_{\beta}(\sigma)$.
From Theorem~\ref{maintheorem} we then get a $P_\beta$-martingale on a DFA that succeeds on the $\beta$-expansion of $x$. Our task in this section is to show that such a martingale can be translated to a base 2 martingale that is computable in polynomial time.
Base 2 suffices because of the (integer) base invariance of polynomial time randomness~\cite{Fi2013}.

The rest of the section is organized as follows. In \S\ref{subsec:change-base} we show a feasible method to approximate dyadic rationals with reals in base $\beta$. In \S\ref{subsec:savings} we introduce the {\em savings property} for $P$-martingales and show that any feasible $P$-martingale can be translated to one with the savings property, preserving the succeeding points. In \S\ref{subsec:measure} we derive some useful properties of the Parry measure $P_\beta$ and introduce the measure $\mu_M$ over $[0,1]$ induced by any $P_\beta$-martingale $M$. In \S\ref{subsec:cdf} we show that the cumulative distribution function of $\mu_M$ is polynomial time computable when restricted to $\beta$-adic inputs. Finally, in \S\ref{subsec:poly-implies-betanormal} we show the main result via an `almost Lipschitz' property, as in~\cite{Fi2013}.

\subsection{Dyadic rationals to base $\beta$}\label{subsec:change-base}

We derive some feasibility properties of $\beta$-ary representation.

\begin{proposition}\label{prop:shift_is_poly}
If $\beta>1$ is Pisot then the set
$
L(X_\beta)=\{\tau\in\Sigma_\beta^*\mid\tau 0^\infty\in X_\beta\}
$ is decidable in linear time.
\end{proposition}

\begin{proof}
Immediate from Theorem~\ref{betacaract}, the fact that $\mathfrak{s}(\beta)$ is eventually periodic (Theorem~\ref{pisotsofic}) and the linear time complexity of lexicographic comparison.
%Recall that $x\in\Sigma_\beta^\NN$ belongs to $X_\beta$ iff for all $n$, $T^n(x)\leqlex \mathfrak{s}(\beta)$. Since $\beta$ is Pisot, $\mathfrak{s}(\beta)$ is periodic and hence given $j$ one can compute $\mathfrak{s}(\beta)\restriction j$ in time $O(|j|)$. Therefore, given $\tau\in\Sigma^*_\beta$, we have that there is $x\in X_\beta$ such that $\tau \prec x$ iff $\tau 000\dots \in X_\beta$ iff  $$(\forall j\in\{1,\dots,|\tau|\})\ \tau\downharpoonright j\leqlex \mathfrak{s}(\beta)\restriction j,$$ and this last condition can be checked in time polynomial in $|\tau|$.
\end{proof}

Fix a finite alphabet $\Sigma$.
We say that a function $g\colon\Sigma^*\to\RR$ is {\em computable} if there is a computable function $\widehat g\colon\Sigma^*\times\NN\to\QQ$ such that for all $\sigma$ and $i$ we have $|\widehat g(\sigma,i)-g(\sigma)|\leq 2^{-i}$. We call $\widehat g$ a {\em computable approximation} of $g$. We say that $g$ is {\em $t(n)$-computable} if there is a Turing machine which on input $\sigma$ and $i$ computes $\widehat g(\sigma,i)$ in time $O(t(i+|\sigma|))$. As usual, we say that $g$ is {\em polynomial time computable} if it is $t(n)$-computable for some polynomial $t$.

\begin{observation}\label{obsering}If $f,g$ are polynomial time computable then $f+g$ and $fg$ are polynomial time computable. %If additionally $f$ and $g$ are bounded then $fg$ is polynomial time computable.\santi{hace falta lo de bounded? nunca lo habia pensado. si es poolytime la salida no puede set muy grande. lo sacaste de algun lado? podemos poner referencia? RTA: No hace falta en el sentido de que es cond suficiente, no necesaria, pero es una condición suficiente y bien boba que no requiere meter proof. Hubiera querido relajarla pero no encontré una relajación suficientemente obvia. Fijate que en general no deberia valer, si queres explico por chat/mail. No, no hay referencias, ker i ko trabajan sobre un intervalo cerado así que todo está acotado}
\end{observation}

\newcommand{\rea}[1]{\langle #1 \rangle_\beta}
\newcommand{\reabin}[1]{\langle #1 \rangle_2}

For $\beta>1$, let $\rea{\cdot}\colon\Sigma_\beta^*\to\RR$ be the function
$
\rea \tau=\sum_{k=1}^{|\tau|} \tau(k-1)\beta^{-k}
$. Observe that in case $\tau$ is a prefix of some sequence in $X_\beta$ then $\rea\tau$ is the only real $x\in[0,1)$ such that
$p_\beta(x)=\tau 0^\infty$.

\begin{proposition}\label{cor:rea_is_poly}
If $\beta>1$ is Pisot, then the function function $\rea{\cdot}$ is polynomial time computable.
\end{proposition}
\begin{proof}The number of summands in $\rea\tau$ is the length of $\tau$, which is computable in linear time. For each summand, $\tau(k)$ is computable in linear time, and $\beta$ is an algebraic number, which is computable in polynomial time~\cite[Corollary 4.3.1]{Ko82}. Since numbers computable in polynomial time form a field~\cite[Corollary 4.3.2]{Ko82}, $\beta^{-k}$ is computable in polynomial time. Then  both $\tau(k)$ and $\beta^{-k}$ are polynomial-time computable and Observation~\ref{obsering} applies.
\end{proof}

Given a real $r\in[0,1)$ and $i\in\NN$, a word $\tau\in L(X_\beta)$ is said to be an {\em approximation of $r$ in base $\beta$ with error $2^{-i}$} if $|\rea{\tau}-r|\leq 2^{-i}$.

\begin{proposition}\label{prop:approx}
If $\beta>1$ is Pisot then the problem of finding an approximation in base $\beta$ of a dyadic rational $\reabin\sigma$ ($\sigma\in \{0,1\}^*$) with error $2^{-i}$ is computable in time polynomial in $|\sigma|+i$.
\end{proposition}

\begin{proof}
Let $\rea{\cdot,\cdot}\colon\Sigma_\beta^*\times\NN\to\QQ$ be a polynomial time computable approximation of $\rea{\cdot}\colon\Sigma_\beta^*\to\RR$ (which exists by Proposition~\ref{cor:rea_is_poly}). Consider Algorithm~\ref{alg:approx}.

\begin{figure}\center
\begin{minipage}[c]{12cm}
    \begin{algorithm}[H] \small
    \caption{Approximation of a dyadic rational in base $\beta$}
    \label{alg:approx}
    \SetKwInOut{Input}{input} \SetKwInOut{Output}{output}

    \Input{$\sigma\in \{0,1\}^*$, $i\in\NN$}

    \Output{$\tau$, a prefix of some sequence in $X_\beta$, such that $|\rea{\tau}-\reabin\sigma|\leq 2^{-i}$} \BlankLine

    let $r=\reabin{\sigma}$ and $\tau=\emptyset$

    \While{$|\rea{\tau,i+2}-r|> 2^{-i-1}$} {

            %let $S =  \{a\in\Sigma_\beta \mid (\exists x\in X_\beta)\ \tau a\prec x\}$
            let $S =  \{a\in\Sigma_\beta \mid \tau a0^\infty \in X_\beta\}$

            let $b\in S$ be the greatest such that $\rea{\tau b,i+2}\leq r$

            \eIf{$b=\max S$}
                {$\tau = \tau b$}
                {
                $b'=b+1$

                \eIf{$\rea{\tau b',i+2}-2^{-i-1}<r$}
                    {$\tau = \tau b'$}
                    {$\tau = \tau b$}
                }
    }
    \end{algorithm}
\end{minipage}
\end{figure}

Notice that when the algorithm terminates, we have $|\rea{\tau,i+2}-r|\leq 2^{-i-1}$; since $|\rea{\tau,i+2}-\rea{\tau}|\leq 2^{-i-2}<2^{-i-1}$, we have $|\rea{\tau}-r|\leq 2^{-i}$. Observe also that
by construction, $\tau$ is always a prefix of some sequence in $X_\beta$. Hence the value of $\tau$ by the time the algorithm terminates satisfies the postcondition. After each execution of the loop body, either
\begin{enumerate}
\item\label{alg:1} $|\rea{\tau,i+2}-r| \leq 2^{-i-1}$ (in which case it will immediately terminate), or
\item\label{alg:2} $\tau \prec p_\beta(r)$.
\end{enumerate}
Let $I_\tau=\{x\in[0,1)\mid \tau\prec p_\beta(x)\}$.
If~\ref{alg:1} does not hold then, by construction, $r\in I_\tau$, and it is clear that in case $\lambda I_\tau\leq 2^{-i-2}$ then it terminates (since $|\rea{\tau,i+2}-\rea{\tau}|\leq 2^{-i-2}$ and $|\rea\tau-r|\leq\lambda I_\tau\leq 2^{-i-2}$, and so $|\rea{\tau,i+2}-r|\leq 2^{-i-1}$).
At each iteration the string $\tau$ is extended in one symbol. We will later see (Corollary~\ref{eq:bound_L}) that $\lambda I_\tau\leq \beta^{-|\tau|}$, if $\beta^{-|\tau|}\leq 2^{-i-2}$ then the algorithm terminates, and so $|\tau|$ is $O(i)$. By
Proposition~\ref{prop:shift_is_poly} and Proposition~\ref{cor:rea_is_poly}, the execution of a single iteration is polynomial in $|\sigma|+i+|\tau|$. Since both the number of iterations and $|\tau|$ is $O(i)$, the  execution of  Algorithm~\ref{alg:approx} on in input $\sigma,i$ is also polynomial in  $|\sigma|+i$.
\end{proof}

\subsection{The savings property}\label{subsec:savings}

We say that a $P$-martingale $M$ on $L\subseteq\Sigma^*$ has the {\em savings
property} if there is $c>0$ such that for all $\tau,\sigma\in
L$, if $\tau\succeq\sigma$ then $M(\sigma)-M(\tau)\leq c$.

\begin{proposition}\label{prop:savings-property}
Let $L\subseteq\Sigma^*$ be a nonempty, factorial and prolongable language, let $P$ be an $L$-supported probability measure on $\Sigma^\NN$ such that there is $a>0$ such that $1-P(\sigma b|\sigma)\leq a\cdot P(\sigma b|\sigma)$ for all $\sigma b\in L$, and let $M$ be a $P$-martingale on $L$ with the savings property via $c$. Then
$M(\sigma)\leq c\cdot d\cdot |\sigma|+M(\emptyset)$ for all $\sigma\in L$.
\end{proposition}

\begin{proof}
As in \cite[Proposition]{Fi2013}, proof is by induction on the length of $\sigma$. For the inductive step, we only need notice that $P(\sigma b|\sigma)$ is well defined and positive. Then
\begin{align*}
M(\sigma b) &= \frac{M(\sigma)- \sum_{d\in \Sigma,d\not=b}P(\sigma d|\sigma)M(\sigma d)}{P(\sigma b|\sigma) }\\
&\leq \frac{M(\sigma)- \sum_{d\in \Sigma,d\not=b}P(\sigma d|\sigma)(M(\sigma)-c)}{P(\sigma b|\sigma) } \tag{$M$ has the savings property}\\
%&=\frac{M(\sigma)- (M(\sigma)-c)(1-P(\sigma b|\sigma))}{P(\sigma b|\sigma)}\\
&=M(\sigma)+c\cdot \frac{1-P(\sigma b|\sigma)}{P(\sigma b|\sigma)}\leq M(\sigma)+c\cdot a\\
&\leq c\cdot a\cdot |\sigma|+M(\emptyset)+c\cdot a=c\cdot a\cdot |\sigma b|+M(\emptyset). \tag{inductive hypothesis}
\end{align*}
This concludes the proof.
\end{proof}

\begin{lemma}[Polynomial time bounded savings property]\label{lem:convert-to-savings-property}
For each polynomial time computable $P$-martingale $N$ there is a polynomial time computable $P$-martingale $M$ which has the savings property and succeeds on all the sequences that $N$ succeeds on.
\end{lemma}

\begin{proof}
The proof of \cite[Lemma 6]{Fi2013} basically works in this case. The only difference is that here $N$ is real-valued instead of rational-valued. This fact is irrelevant for the polynomial time bound. One can verify that the same definition of $M$ as in \cite[Lemma 6]{Fi2013} yields a polynomial time $P$-martingale.
\end{proof}

\subsection{The measure induced by $P_\beta$-martingales}\label{subsec:measure}

Recall that $p_\beta$ is the one-to-one mapping that sends each real in $[0,1)$ to its unique $\beta$-expansion, and that $\widehat P_\beta$ is the Parry measure induced on the unit interval, i.e.\ $\widehat{P}_\beta=P_\beta\circ p_\beta$.
Let $T_\beta\colon[0,1]\to[0,1)$ be the map $T_\beta(x)=\{\beta x\}$.

We derive some useful properties of the Parry measure.
\begin{theorem}\cite{Parry60} Let $\beta>1$ be a real base, then the Parry measure $P_{\beta}$ is $L(X_\beta)$-supported.
\end{theorem}

%From Theorem~\ref{teoparry} we know the following closed expression for the measure that $P_\beta$ induces on the unit interval\santi{Ok con sacar la ecuacion y referenciar al thm~\ref{teoparry} (poniendo nro. a la ecuacion de ahi? RTA: ok}
%%
%\begin{equation*}\label{parrymeasure}
%\widehat{P}_\beta([a,b])=\int_a^b\sum_{n=1}^\infty\mathbbm{1}_{[0,T_\beta^n(1))}(x)\frac{1}{\beta^n}dx.
%\end{equation*}
%%
%Similarly, w
We will use
\begin{equation*}%\label{eqn:L}
\pfl=\lambda\circ p_\beta^{-1}
\end{equation*}
for the push-forward of the Lebesgue measure on the $\beta$-shift. Of course, inequality~\eqref{eqn:bound_hatP} translates to
\begin{equation}\label{eq:bound_P}
k'\cdot \pfl(\sigma)\leq P_\beta(\sigma)\leq k\cdot \pfl(\sigma)
\end{equation}
for any $\sigma\in L(X_\beta)$.
Let us say that $x\in[0,1]$ is $\beta$-adic if $x$ has a finite $\beta$-expansion. Clearly, $\beta$-adic numbers correspond bijectively to words in $L(X_\beta)$ not ending in 0. For instance, if $\beta$ is 2.5 we have that both $2/5$ and $24/25$ are $\beta$-adic numbers, since their fractional $\beta$-expansions are $p_\beta(2/5)=10^\infty$ and $p_\beta(24/25)=210^\infty$.

We will write $I_\sigma$ for the interval of real numbers in $(0,1)$ whose fractional $\beta$-expansion begins with $\sigma$. Observe that if $\sigma\notin L(X_\beta)$ then $I_\sigma=\emptyset$.

Since we will be working with the Parry measure and with $\beta$-expansions, and since the Parry measure has a closed expression in terms of Lebesgue measure, it will be helpful to know what kind of values $\lambda(I_\sigma)$ can take, given that in the non-integer case it is no longer true that $\lambda(I_\sigma)= \beta^{-|\sigma|}$.

For this purpose we introduce some new notation. For $\sigma\in L(X_\beta)$, write
\begin{align*}
{\sf Suc}_1(\sigma)&=\{b\in\Sigma_\beta\mid\ \sigma b\in L(X_\beta)\},
\\
\overline{\sigma}^+&=\max{\sf Suc}_1(\sigma),
\\
\text{next}(\sigma)&=\min_{\leq_{\text{lex}}}\{\tau\in L(X_\beta)\colon \sigma<_{\rm lex} \tau\},
\\
\mathcal{L}&=\{\sigma b\in L(X_\beta)\colon b\neq\overline{\sigma}^+\}.
\end{align*}
Notice that, because of Theorem~\ref{betacaract},  ${\sf Suc}_1(\sigma)$ has the form $\{1,\dots,r\}$ for some $r= \overline{\sigma}^+\leq\lceil\beta\rceil-1$. Also, given any $b\in\Sigma$ we have that $\sigma b\in \mathcal{L}$ if and only if some suffix of $\sigma$ is a prefix of $\mathfrak{s}(\beta)$.
%
%\santi{te parece poner este parrafo  y el lem~\ref{lemaexp} mucho antes, cuando se define el overline de words? RTA: ok}

Let us make a remark concerning $\overline{s\uh{i}}$ for some $\beta$-expansion~$s$.
When $\beta$ is an integer and $x\in(0,1)$, if $I_\beta^n(x)=[a,b)$ denotes the $\beta$-adic half-open interval of measure $\beta^{-n}$ that $x$ lies in,  then the sequence $(I_\beta^n(x))_{n\in\NN}$ cannot eventually consist of the rightmost $\beta$-adic subinterval of the previous $\beta$-adic interval. In terms of its $\beta$-expansion, it cannot eventually consist of an infinite sequence of $\beta-1$, since the rules for the construction of $\beta$-expansions mandate that $\dots a(\beta-1)^\infty\ (a<\beta -1)$ be written $\dots(a+1)0^\infty$. The same observation is true for non-integer bases $\beta$, when the symbol identifying the rightmost $\beta$-adic subinterval of a $\beta$-adic interval is not necessarily $\lfloor\beta\rfloor-1$. In this case, $\overline{\sigma}^+$ is used to identify the rightmost $\beta$-adic subinterval of $I_\sigma$, i.e.\ $I_{\sigma\overline{\sigma}^+}$, and our observation takes the following form.
\begin{lemma}\label{lemaexp}Let $s$ be the fractional $\beta$-expansion of some real $x\in[0,1)$ (that is, $s\in p_\beta([0,1))$). Then for any natural number $n$, there is $i>n$ such that
$s_{i+1}\neq \overline{s\uh{i}}^+$.
\end{lemma}

\begin{observation}\label{obse1}
$\pfl(\sigma)=\rea{\text{next}(\sigma)}-\rea{\sigma}$.
\end{observation}

\begin{lemma}\label{lema59}
Let $\sigma=\sigma' b\in \mathcal{L}$.
Then $\pfl(\sigma)=\beta^{-|\sigma|}$.
\end{lemma}
\begin{proof}Since $b\neq \overline{\sigma}'^+$ we have that $\sigma' (b+1)\in L(X_\beta)$, so that $\text{next}(\sigma)=\sigma' (b+1)$ and by Observation~\ref{obse1}, $\lambda (I_\sigma)=\pfl(\sigma)=\rea{\sigma' (b+1)}-\rea{\sigma' b}=\beta^{-(|\sigma'|+1)}$
\end{proof}

\begin{observation}\label{obse2}Let $\tau c\in \mathcal{L}$ and $\alpha$ be some prefix of $\mathfrak{s}(\beta)$. Then $\text{next}(\tau c\alpha)=\tau (c+1)$
\end{observation}

Define
$
N_\sigma=\min\{n\colon \sigma_{n+1}\dots \sigma_{|\sigma|} = \mathfrak{s}(\beta)_1\dots\mathfrak{s}(\beta)_{|\sigma|-n}\}\cup\{|\sigma|\}.
$
\begin{observation}\label{obse3}
Let $\sigma\in L(X_\beta)$, $N_\sigma >1$ and $\tau c=\sigma\uh{N_\sigma}$, then $c\neq \overline{\tau}^+$.
\end{observation}
The following lemma extends Lemma~\ref{lema59} when $\beta$ is Pisot, in the sense that it completes the characterization of the values that $\pfl(\sigma)$ may take.
\begin{lemma}\label{lemapulenta}
Let $\beta$ be a Pisot real, let $\sigma\in L(X_\beta)\setminus \mathcal{L}$, and suppose $\mathfrak{s}(\beta)=r_1\dots r_m(a_1\dots a_n)^\infty$ and $\phi_j=\rea{a_1\dots a_j}$ for $j\leq n$. Let $\tau c=\sigma\uh{N_\sigma}$. Then, $\pfl(\sigma)=\pfl(\tau c)\pfl(r_1\dots r_l)$ in case $\sigma=\tau c r_1\dots r_l$, $l\leq m$), or $\pfl(\sigma)=\pfl(\tau c)\pfl(r_1\dots r_m a_1\dots a_k)\beta^{-ln}$ in case $\sigma=\tau c r_1\dots r_l(a_1\dots a_n)^la_1\dots a_k$, $k\leq n$, $0\leq l$.
%
%
%$$\frac{\pfl(\sigma)}{\pfl(\tau c)}=\pfl(r_1\dots r_l)\quad\text{(in case $\sigma=\tau c r_1\dots r_l$, $l\leq m$)}$$
%or
%$$\frac{\pfl(\sigma)}{\pfl(\tau c)}=\pfl(r_1\dots r_m a_1\dots a_k)\beta^{-ln}$$
%(in case $\sigma=\tau c r_1\dots r_l(a_1\dots a_n)^la_1\dots a_k$, $k\leq n$, $0\leq l$)
\end{lemma}
\begin{proof}Since $\sigma\notin\mathcal{L}$ some suffix of $\sigma$ is a prefix of $\mathfrak{s}(\beta)$, which means either $\sigma=\tau cr_1\dots r_l$, for some $l\leq m$, or $\sigma=\tau c r_1\dots r_m(a_1\dots a_n)^la_1\dots a_k$, for some $l\geq 0$ and $k\leq n$. Write $\psi_i=\rea{r_1\dots r_i}$. One then has
\begin{equation}\label{unopisot}
1=\rea{\mathfrak{s}(\beta)}=\psi_m+\frac{1}{\beta^{m}}\phi_n\frac{1}{1-\beta^{-n}}.
\end{equation}
In case $\sigma=\tau c r_1\dots r_l$, we have $\rea{\sigma}=\rea{\tau c}+\beta^{-(|\tau|+1)}\psi_l$, and by Observations~\ref{obse1} and~\ref{obse2} we have $\pfl(\sigma)=\rea{\tau (c+1)}-\rea{\sigma}$, so that
$$
\frac{\pfl(\sigma)}{\pfl(r_1\dots r_l)}=\frac{\rea{\tau (c+1)}-\rea{\tau c}-\beta^{-(|\tau|+1)}\psi_l}{1-\psi_l}=\frac{\pfl(\tau c)-\pfl(\tau c)\psi_l}{1-\psi_l}=\pfl(\tau c),
$$
where we have used that $\pfl(\tau c)=\beta^{-(|\tau|+1)}$ (which follows from Observation~\ref{obse3}).

For the case when $\sigma=\tau c r_1\dots r_m(a_1\dots a_n)^la_1\dots a_k$, we have
\begin{align*}
\rea{\sigma}&=\rea{\tau c}+\beta^{-(|\tau|+1)}[\psi_m+\beta^{-m}(\phi_n(1+\beta^{-1}+\dots+\beta^{-(l-1)n})+\beta^{-ln}\phi_k)]\nonumber\\
&=\rea{\tau c}+\beta^{-(|\tau|+1)}\left[\psi_m+\beta^{-m}\left(\phi_n\frac{(\beta^{-ln}-1)}{\beta^{-n}-1}+\beta^{-ln}\phi_k\right)\right],
\end{align*}
and from~\eqref{unopisot} and Observations~\ref{obse1} and~\ref{obse2}
$$
\pfl(r_1\dots r_m a_1\dots a_k)=1-(\psi_m+\beta^{-m}\phi_k)=\beta^{-m}\left[\frac{\phi_n}{1-\beta^{-n}}-\phi_k\right]
$$
so that (using Observations~\ref{obse1} and~\ref{obse2} again)
\begin{align*}
\pfl(\sigma)&=\rea{\tau (c+1)}-\rea{\sigma} \nonumber \\
&=\pfl(\tau c)-\beta^{-(|\tau|+1)}\left[\psi_m+\beta^{-m}\left(\phi_n\frac{(\beta^{-ln}-1)}{\beta^{-n}-1}+\beta^{-ln}\phi_k\right)\right]\nonumber\\
&=\pfl(\tau c)\left[1-(\psi_m+\beta^{-m}\phi_k)-\beta^{-m}(\beta^{-ln}-1)\left(\phi_k-\frac{\phi_n}{1-\beta^{-n}}\right)\right]\nonumber\\
&=\pfl(\tau c)\pfl(r_1\dots r_m a_1\dots a_k)\beta^{-ln}.
\end{align*}
This concludes the proof.
\end{proof}

\begin{corollary}\label{coroconstantes}There exist positive constants $d$ and $d'$ such that
$
d\leq \pfl(\sigma b\mid \sigma)\leq d'
$
for any $\sigma b$ such that $\pfl(\sigma b)>0$.
\end{corollary}
\begin{proof}
It suffices to see that $\pfl(\sigma b\mid \sigma)$ takes only finitely many values.
First of all, in the case when $\sigma b,\sigma\in\mathcal{L}$ then $\pfl(\sigma b\mid \sigma)=\beta^{-1}$, by Lemma~\ref{lema59}.
When $\sigma\in\mathcal{L}$ but $\sigma b\notin\mathcal{L}$ then $\pfl(\sigma b\mid\sigma)=\pfl(b)$, by Lemma~\ref{lemapulenta}.
When $\sigma b\in\mathcal{L}$ but $\sigma \notin\mathcal{L}$, write $\tau c=\sigma\uh{N_\sigma}$. As remarked above, either $\sigma =\tau cr_1\dots r_l$ or $\sigma=\tau c r_1\dots r_l(a_1\dots a_n)^la_1\dots a_k$. Then, by Lemma~\ref{lemapulenta} either
$$
\pfl(\sigma b\mid \sigma)=\frac{\beta^{-(|\tau c|+l+1)}}{\pfl(\tau c)\pfl(r_1\dots r_l)}=\frac{\beta^{-(l+1
)}}{\pfl(r_1\dots r_l)}
$$
(which can take only finitely many values, since $l\leq m$ and $m$ is fixed) or
$$
\pfl(\sigma b\mid \sigma)=\frac{\beta^{-(|\tau c|+m+ln+k+1)}}{\pfl(\tau c)\pfl(r_1\dots r_ma_1\dots a_k)\beta^{-ln}}=\frac{\beta^{-(k+1)}}{\pfl(r_1\dots r_ma_1\dots a_k)}
$$
(which can take only finitely many values, since $k\leq n$, and $n$ is fixed).

When neither $\sigma b$ nor $\sigma$ are in $\mathcal{L}$, Lemma~\ref{lemapulenta} means the conditional probability $\pfl(\sigma b\mid \sigma)$ may take the following values.
\begin{itemize}
\item $\pfl(r_1\dots r_{l+1})/\pfl(r_1\dots r_l)$, if $\sigma=(\sigma\uh{N_\sigma})r_1\dots r_l$ for some $l\leq m-1$

\item $\pfl(r_1\dots r_{m}a_1)/\pfl(r_1\dots r_m)$, if $\sigma=(\sigma\uh{N_\sigma})r_1\dots r_m$

\item $\pfl(r_1\dots r_{m}a_1\dots a_{k+1})/\pfl(r_1\dots r_{m}a_1\dots a_{k})$, if $\sigma=(\sigma\uh{N_\sigma})r_1\dots r_m(a_1\dots a_n)^la_1\dots a_k$ for $0\leq l$, $k\leq n-1$

\item $\pfl(r_1\dots r_{m}a_1)\beta^{-1}/\pfl(r_1\dots r_{m}a_1\dots a_{n})$, if $\sigma=(\sigma\uh{N_\sigma})r_1\dots r_m(a_1\dots a_n)^la_1\dots a_n$ for $0\leq l$
\end{itemize}
Since these expressions may only take finitely many values (for fixed $r_1,\dots,r_m$ and fixed $a_1,\dots,a_n$), the proof is finished.
\end{proof}

\begin{corollary}\label{eq:bound_L}If $\beta$ is Pisot, then
$
\pfl(\sigma)\leq \beta^{-|\sigma|}.
$
\end{corollary}
\begin{proof}
It is enough to note that
$
\pfl(\mathfrak{s}(\beta)\uh{k})\leq \beta^{-k}
$
(this is true for any $\beta$, Pisot or not)
and then use Lemma~\ref{lemapulenta}.
\end{proof}

The following is a straightforward consequence of Proposition~\ref{prop:savings-property} and Corollary~\ref{coroconstantes}:

\begin{corollary}\label{cor:malabia}
If $\beta>1$ is Pisot and $M$ is a $P_\beta$-martingale on $L(X_\beta)$ with the savings property, then there is $c$ such that $M(\sigma)\leq c\cdot |\sigma|+M(\emptyset)$ for all $\sigma\in L(X_\beta)$.
\end{corollary}

\newcommand{\rep}[2]{\langle0.#1\rangle_{#2}}
\newcommand{\cdfM}{\cdf_M}
\newcommand{\cdf}{{\sf cdf}}
\newcommand{\mart}{{\sf mart}}
\newcommand{\voc}{\Sigma}

Let $\beta>1$ be Pisot.
Each $P_\beta$-martingale $M$ on $L(X_\beta)$ induces a measure $\mu_M$ on the
algebra of word cylinders defined by $\mu_M([\sigma])=M(\sigma)\cdot P_\beta(\sigma)$,
for $\sigma\in L(X_\beta)$. Via Carath\'eodory's extension theorem
this measure can be extended to a Borel measure on $\Sigma^\NN$, and
if $\mu_M$ is atomless (i.e.\ no point has positive measure), we can
also think of it as a Borel measure on $[0,1]$, which is given by
%
%\begin{equation*}\label{eqn:cond-mum}
$
\mu_M(I_\sigma)=M(\sigma)\cdot P_\beta(\sigma),
$
%\end{equation*}
%

We say that a martingale $M$ is {\em atomless} if $\mu_M$ is atomless.

\begin{observation}
If $M$ is a $P_\beta$-martingale with the savings property then it is atomless.
\end{observation}
\begin{proof}
Indeed, by~\eqref{eq:bound_P} and Corollary~\ref{eq:bound_L}, there is a constant $k$ such that for any $\sigma\in
L(X_\beta)$, $P_\beta(\sigma)\leq k\cdot \beta^{-|\sigma|}$. By
Corollary~\ref{cor:malabia}, there is a constant $c$ such that for any $\sigma\in
L(X_\beta)$ of length $n$ we have $\mu_M(I_\sigma)\leq k\cdot \beta^{-n} \cdot (d\cdot n + M(\emptyset))$,
and this goes to $0$ as $n$ goes to
infinity. Hence $\mu_M$ is atomless.
\end{proof}
The cumulative distribution function associated with $\mu_M$ will be written $\cdfM(x)$ for $x\in[0,1]$.
We now want to prove an analogue of the left-to-right implication of Theorem 3.6 from~\cite{Bratt11} (which is used as Theorem 3 in~\cite{Fi2013}).
\begin{theorem} \label{thm:succeeds-then-slope-diverges}
Let $M$ be a $P_\beta$-martingale with the savings property that succeeds on the $\beta$-expansion of $z\in (0,1)$, a non-$\beta$-adic real. Then
\begin{equation*}%\label{limitinf}
\liminf_{h\to 0}\frac{\cdfM(z+h)-\cdfM(z)}{h}=\infty.
\end{equation*}
\end{theorem}

\begin{proof}
The proof is essentially the same as in~\cite{Bratt11}. Let $g=\cdfM$ and let $r>0$. We will show that there is some $\epsilon>0$ such that if $|h|<\epsilon$ then $|g(z+h)-g(z)|>rk'|h|$, where $k'$ is the constant such that $k'\lambda(A)\leq \widehat{P}_\beta(A)$ from~\eqref{eqn:bound_hatP}.

Let $(z^\beta_{i})_{i\geq 1}$ be the fractional $\beta$-expansion of $z$. Since $M$ succeeds on $(z_{\beta,i})_{i\geq 1}$ and has the savings property, there is some $i$ such that, if $\rho=z^\beta\uh{i}$, then $M(\rho\tau)>r$ for any $\tau$ such that $\rho\tau\in L(X_\beta)$. Since $z$ is not $\beta$-adic there is some $j>i$ such that $z^\beta_j\neq0$, and by Lemma~\ref{lemaexp} there is some $k>j$ such that $z^\beta_{k+1}\neq \overline{z^\beta\uh{k}}^+$. Let $\epsilon=\beta^{-k-1}$. If $0<|h|<\epsilon$ then the $\beta$-expansion of $z+h$ extends $\rho$. If $h>0$ this is because $z+h<z+\beta^{-k-1}$ and $\beta^{k+1}$ has the same $\beta$-expansion as $z$, except that $z^\beta_{k+1}$ is replaced with $1+z^\beta_{k+1}$, which at worst is $\overline{z^\beta\uh{k}}^+$. Similarly, if $h<0$, then $z+h>z-\beta^{-k-1}>z-\beta^{-j}$ and $z-\beta^{-j}$ has the same $\beta$-expansion as $z$, except that $z^\beta_j$ is replaced with $z^\beta_j-1$, which at worst is 0. This means that, if $W\subseteq L(X_\beta)$ is a prefix-free set of strings such that $\bigcup_{\sigma\in W}I_\sigma=(z,z+h)$ in case $h>0$, or $\bigcup_{\sigma\in W}I_\sigma=(z+h,z)$ in case $h<0$, then all strings in $W$ extend $\rho$. Hence,
$$
|g(z+h)-g(z)|=\sum_{\sigma\in W}M(\sigma)P_\beta(\sigma)>rk'\sum_{\sigma \in W}\pfl(\sigma)=rk'\sum_{\sigma\in W}\lambda(I_\sigma)=rk'|h|.
$$
\end{proof}

In~\cite{Bratt11} it is shown that if $f$ is a nondecreasing function with domain
containing $[0,1]\cap\QQ$ then
$\mart_f\colon\{0,1\}^*\to\RR$ defined by
$$
\mart_f(\tau)=\frac{f(\reabin{\tau}+2^{-|\tau|})-f(\reabin{\tau})}{2^{-|\tau|}}.
$$
is a classical martingale. It is also observed in \cite[Fact 3.5]{Bratt11} that if $f(0) = 0$. Then $\cdf_{\mart_f} = f$. In the next lemma we use these facts for $f=\cdfM$, the cumulative distribution function of our $P_\beta$-martingale.

\begin{lemma}\label{lem:martingales-base-conversion}
Let $\beta>1$ be Pisot. Suppose $M$ is a $P_\beta$-martingale with the savings property. Let $N\colon\{0,1\}^*\to\RR_{\geq 0}$ be the
following (classical) martingale:
\begin{equation*}%\label{eqn:N}
N(\tau)=\mart_{\cdfM}(\tau)=\frac{\cdfM(\reabin{\tau}+2^{-|\tau|})-\cdfM(\reabin{\tau})}{2^{-|\tau|}}.
\end{equation*}
Suppose $s\in X_\beta$ and that there exists $x\in[0,1]$ neither $\beta$-adic, nor a dyadic rational such that $p_\beta(x)=s$. If $M$ succeeds
on $s$ then $N$ succeeds on the fractional binary expansion of $x$.
\end{lemma}

\begin{proof} Same proof as in \cite[Lemma 11]{Fi2013}, with Theorem~\ref{thm:succeeds-then-slope-diverges} substituting for \cite[Theorem 10]{Fi2013}
\end{proof}
%\begin{proof}
%Since $x$ is not $\beta$-adic, $s$ is its unique $\beta$-expansion, and since $x$ is not a dyadic rational it has a unique binary expansion.

%If $M$ succeeds on $s$, by Theorem~\ref{thm:succeeds-then-slope-diverges} we have that condition (\ref{limitinf}) is true. By \cite[Fact 3.5]{Bratt11}, $\cdfM=\cdf_N$, and hence condition (\ref{limitinf}) is true replacing $M$ by $N$. By \cite[Theorem 3.6]{Bratt11} $N$ succeeds on the fractional binary expansion of $x$.
%\end{proof}

\subsection{\boldmath{$\mu_M$} and \boldmath{$\cdfM$} are polynomial time computable}\label{subsec:cdf}

As in~\cite{Fi2013}, we show an `almost Lipschitz' condition for $\cdfM$:
\begin{proposition}\label{cor:lipschitz}
Let $\beta>1$ be Pisot and let $M$ be a $P_\beta$-martingale on $L(X_\beta)$ with the savings property.  Then
there are constants $k,\epsilon>0$ such that for every
$x,y\in[0,1]$, if $y-x\leq\epsilon$ then
$$
\cdfM(y)-\cdfM(x)\leq -k\cdot(y-x)\cdot\log(y-x).
$$
\end{proposition}

\begin{proof}
We actually show that there are constants $c$ and $d$ such that

\begin{equation}\label{eq:lipschitz}
\cdfM(y)-\cdfM(x)  \leq
            d\cdot (y-x)\cdot (c\cdot (1-\log_\beta(y-x))
            +M(\emptyset))
\end{equation}
for $0\leq x<y\leq 1$. Let $n\in \NN$ be the
least integer such that $\beta^{-n}<y-x$, and let
$$
\Theta=\{\sigma\in L(X_\beta) \mid |\sigma|=n, I_\sigma \cap [x,y]\not=\emptyset\}.
$$
So we may write $\Theta=\{\sigma_1,\dots,\sigma_m\}$, with $\sigma_i<_{\rm lex}\sigma_{i+1}$ for all $i<m$. Let $p$ be the left end-point of $I_{\sigma_1}$ and let $q$ the right end-point of $I_{\sigma_m}$. Clearly both $p$ and $q$ are $\beta$-adic and $[x,y]\subseteq\bigcup_{\sigma\in\Theta}I_\sigma=[p,q]$.

We have
\begin{align*}
\cdfM(y)-\cdfM(x)   &\leq  \cdfM(q)-\cdfM(p)= \mu_M [p,q]=\sum_{\sigma\in \Theta}P_\beta(\sigma)\cdot M(\sigma)\\
            &\le (c\cdot n+M(\emptyset))\sum_{\sigma\in \Theta}P_\beta(\sigma)\tag{by Corollary~\ref{cor:malabia}}\\
            &\le \left((c\cdot (1-\log_\beta (y-x))\right)+M(\emptyset))\sum_{\sigma\in \Theta}P_\beta(\sigma)\tag{$\beta^{-(n-1)}\geq y-x$}.
\end{align*}
Since for each $\sigma\in\Theta$ we have $P_\beta(\sigma)=\widehat P_\beta(I_\sigma)$, from~\eqref{eqn:bound_hatP} we conclude $P_\beta(\sigma)\leq k\cdot\lambda(I_\sigma)$, and so $\sum_{\sigma\in \Theta}P_\beta(\sigma)\leq k\cdot\lambda([x,y])+P_\beta(\sigma_1)+P_\beta(\sigma_m)$. From~\eqref{eq:bound_P} and Corollary~\ref{eq:bound_L} we know that $P_\beta(\sigma_1)$ and $P_\beta(\sigma_m)$ are at most $\beta^{-n}<y-x)$. Hence we conclude~\eqref{eq:lipschitz} for $d=k+2$.
\end{proof}

\begin{proposition}\label{prop:P-polytime}
If $\beta>1$ is Pisot then $P_\beta$ is polynomial time computable.
\end{proposition}

\begin{proof}
%Recall that\santi{esto ya lo dijimos (dos veces) antes. Ok con sacar la ecuacion y referenciar al thm~\ref{teoparry}? RTA: ok}
%$$
%\widehat{P}_\beta([a,b])=\int_a^b\sum_{n=0}^\infty\mathbbm{1}_{[0,T_\beta^n(1))}(x)\frac{1}{\beta^n}dx
%$$
Recall from \eqref{eqn:teoparry} the closed expression for $\widehat{P}_\beta$, the measure that $P_\beta$ induces on the unit interval.
As before, since $\beta$ is Pisot its periodic $\beta$-expansion can be written $\mathfrak{s}(\beta)=r_1\dots r_m(a_1\dots a_n)^\infty$.

Define
$$
B_k=\{\tau\in \Sigma_\beta^*\colon \tau\leqlex T^k(\mathfrak{s}(\beta))\}.
$$
Notice that $B_k$ is computable in linear time. Since $x\leq T^n_\beta(1)$ iff $p_\beta(x)\leqlex T^n(\mathfrak{s}(\beta))$ we have, for any $\sigma\in L(X_\beta)$,
$$
\widehat{P}_\beta(I_\sigma)=P_\beta(\sigma)=\sum_{j=0}^{m-1}\frac{\pfl(I_\sigma)}{\beta^j}\mathbbm{1}_{B_j}(\sigma)+\frac{1}{\beta^{m-1}}\sum_{j=1}^n\frac{\pfl(I_\sigma)}{\beta^j-1}\mathbbm{1}_{B_{j+m-1}}(\sigma),
$$
which is polynomial time computable because it consists of fixed-length sums of products of polynomial time computable functions, since $\pfl(\cdot)$, $\beta$ and the set $B_k$ are polynomial time computable.
\end{proof}

\begin{proposition}\label{prop:cdf-polytime}
Let $\beta>1$ be Pisot, and let $M$ be a polynomial time computable $P_\beta$-martingale with the savings property. Then both $\mu_M\colon L(X_\beta)\to\RR$ and $f\colon L(X_\beta)\to\RR$ given by $f(\sigma)=\cdfM(\rea{\sigma})$ are computable in polynomial time.
\end{proposition}

\begin{proof}
We have $\mu_M(\sigma)=M(\sigma)P_\beta(\sigma)$.
By Proposition~\ref{prop:P-polytime} $P_\beta$ is computable in polynomial time and $M$ is also computable in polynomial time by hypothesis. So that by Observation~\ref{obsering} their product is computable in polynomial time.

 %Now, $M$ is not bounded so we cannot use Observation~\ref{obsering}. However, since $M$ has the savings property it is bounded by a linear function, and given a polynomial time computable approximation $\widehat{P}(\sigma,i)$ of $P_\beta$ we may find a polynomial time computable approximation $\tilde{P}(\sigma,i)$ such that $|P_\beta(\sigma)-\tilde{P}(\sigma,i)|<2^{-1}|\sigma|^{-1}$ (namely, $\tilde{P}(\sigma,i)=\widehat{P}(\sigma,i+i'(\sigma))$ where $i'(\sigma)=\lfloor \log_2(|\sigma|)\rfloor+1$), and then we may take $\hat{\mu}(\sigma,i)\tilde{P}(\sigma,i)\hat{M}(\sigma,i)$ for $\hat{M}$ some polynomial time computable approximation of $M$. Clearly, $\hat{\mu}$ is computable in polynomial time and $|\mu_M(\sigma)-\hat{\mu}(\sigma,i)|<K2^{-i}$ for a constant $K$.

For $f$ we have
$$
f(\sigma)=\mu_M([0,\rea{\sigma}])=\sum_{i=0}^{|\sigma|-1}\sum_{b\in {\sf Suc}_1(\sigma\uh{i})}\mu_M((\sigma\uh{i}) b)\nonumber%\\
%&=&\sum_{i=0}^{|\sigma|-1}\sum_{b\in {\sf Suc}_1(\sigma\uh{i})} M((\sigma\restriction{i}) b)\cdot P_\beta((\sigma\restriction{i}) b)\nonumber
$$
By Proposition~\ref{prop:shift_is_poly} membership in ${\sf Suc}_1(\sigma)$ is checked in linear time and we have a sum of at most $|\sigma|\cdot(1+\lfloor\beta\rfloor)$ many terms, each of which is computable in polynomial time.
\end{proof}

\subsection{Polynomial time randomness implies normality to Pisot bases}\label{subsec:poly-implies-betanormal}

\begin{lemma}\label{lemautil}Let $\beta$ be a Pisot number and $M$ be a $P_\beta$-supermartingale that is computable in polynomial time and succeeds on $s\in X_\beta$. Then there is a $P_\beta$-martingale $\widehat{M}$ computable in polynomial time that succeeds on $s$.
\end{lemma}
\begin{proof}Same as in \cite[Lemma 4]{Fi2013}. Define $d(\sigma)=M(\sigma)-P(\sigma)^{-1}\sum_{a\in\Sigma_\beta}P(\sigma a)M(\sigma a)$ for any $\sigma\in L(X_\beta)$. Notice that $P(\sigma a)M(\sigma a)$ is computable in polynomial time by the same argument used to show in Proposition~\ref{prop:cdf-polytime} that $\mu_M$ is computable in polynomial time. Then $\widehat{M}(\sigma)=M(\sigma)+\sum_{\tau\prec\rho}d(\sigma)$ is a $P_\beta$-martingale computable in polynomial time.
\end{proof}

\begin{lemma}\label{lem:not-normal-then-p-mart-is-poly}
Let $\beta>1$ be Pisot. If $s\in X_\beta$ is not $P_\beta$-distributed then there is a polynomial time computable $P_\beta$-martingale which succeeds on $s$.
\end{lemma}
\begin{proof}
We know from Theorem~\ref{pisotsofic} that $X_\beta$ is a sofic subshift. Moreover, by Theorem~\ref{betacaract} it is clearly irreducible (if $\sigma,\tau\in L(X_\beta)$ then $\sigma 0\tau\in L(X_\beta)$). Let $(G,\mathcal{L})$ be its minimal, irreducible, right-resolving presentation, whose existence is guaranteed by Theorem~\ref{teolind}) and let $\mathcal{L}^*\colon X_G\to X_\beta$ be the natural factor map induced by $\mathcal{L}$. Of course, $X_G$ is an irreducible, 1-step Markov shift and it is shown in \cite[Example 8.1.6]{Lind95} that $\mathcal{L}^*$ is finite-to-one\footnote{I thank Mike Boyle and Brian Marcus for pointing out this argument to me.}, that is, for any $s\in X_\beta$, $\mathcal{L}^{*-1}(s)$ is a finite set. It is also shown in \cite[Corollary 8.1.20]{Lind95} that finite-to-one factor maps between irreducible sofic shifts preserve topological entropy, where the topological entropy $h(X)$ of a subshift $(X,T)$ is defined as\footnote{The topological entropy of a dynamical system has a different, intrinsic definition not depending on measures, and our ``definition'' is actually a theorem known as the \textit{variational principle}.}
$$
h(X)=\sup_{T\text{-invariant }\mu}h_\mu(X).
$$
Thus, $h(X_G)=h(X_\beta)$. But from Theorem~\ref{teoparry2} we know that there is a unique invariant $P$ such that $h(X_G)=h_P(X_G)$ where $P$ is an irreducible, 1-step Markov measure on $X_G$. Furthermore, from Theorem~\ref{teoparry} we know that $h(X_\beta)=h_{P_\beta}(X_\beta)$. Now, it is shown in \cite[Theorem 1.1]{Coven74} that for any factor map between subshifts $\pi\colon X\to Y$ and any invariant measure $\nu$ on $Y$, there is an invariant measure $\mu$ on $X$ such that $\nu=\mu\circ\pi^{-1}$. Hence, there is some invariant $\mu$ on $X_G$ such that $P_\beta=\mu\circ\mathcal{L}^{*-1}$ and since factor maps cannot increase metric entropy (\cite[Theorem 4.11]{Walters95}) we have that $h_\mu(X_G)\geq h_{P_\beta}(X_\beta)=h(X_\beta)=h(X_G)=h_P(X_G)$. But since $P$ is the unique invariant measure of maximal metric entropy on $X_G$, it follows that $P=\mu$, and therefore $P_\beta$ is the push-forward of $P$. Thus, we are under the conditions of
Theorem~\ref{teosofic} and we have, according to item~\ref{teosofic-1}, a $P_\beta$-martingale generated by a DFA which succeeds on~$s$, and, according to item~\ref{teosofic-2}, a $P_\beta$-supermartingale generated by a DFA which succeeds on~$s$. In any case, we have a $P_\beta$-supermartingale generated by a DFA which succeeds on $s$ and whose two betting factors other than 1 are either rational or have the form $(1-\delta p^*/(1-p^*))$, where $p^*$ is the conditional $P_\beta$ probability on some fixed words. Since $P_\beta$ is polynomial time computable, then all betting factors are polynomial time computable.

Thus, our $P_\beta$-supermartingale is of the form $L(\sigma)=p^{m_1(\sigma)}r^{m_2(\sigma)}$, where $r$ is polynomial time computable, $p$ is some fixed rational and $m_1(\sigma)$ and $m_2(\sigma)$ are non negative integers smaller than $|\sigma|$ and computable in time linear in $|\sigma|$, since a DFA reads its entry in linear time.

For rational $p$ it is clear that $p^{m_1(\sigma)}$ is polynomial time computable. Now, when $r$ is not rational, $r$ is strictly smaller than 1, and is also computable in polynomial time. So, given $n$, we can compute a rational $r_n$ in time $O(q(n))$ ($q$ some polynomial) such that $|r_n-r|<2^{-n}$. Then, if $\epsilon_n=r-r_n$, we have, for $m=m_2(\sigma)$,
$$
|r^{m}-r_n^{m}|\leq |r^{m}|+|(r+\epsilon_n)^{m}|\leq 2\sum_{i=1}^{m} \binom{m}{i}|\epsilon_n|^ir^{m-i}\leq 2^{m}|\epsilon_n|=2^{m-n}.
$$
Thus, given $|\sigma|$ and $k$, we can compute a rational $r_n$ in time $O(q(n))$ for $n=m_2(\sigma)+k+1\leq |\sigma|+k+1$, and we can compute $m_2(\sigma)$ in time $O(q'(|\sigma|)$ for some polynomial $q'$. Therefore, since exponentiation by squaring has (strictly less than) polynomial time complexity the number $r(\sigma,k)=r_n^{m_2(\sigma)}$ can be computed in $O(q''(|\sigma|+k))$ time for some polynomial $q''$, and satisfies
$
|r^{m_2}-r(\sigma,k)|<2^{-k}.
$
Hence, $r^{m_2(\sigma)}$ is computable in polynomial time and so is the $P_\beta$-supermartingale $L$.
By Lemma~\ref{lemautil} there is a polynomially time computable $P_\beta$-martingale that succeeds on $s$.
\end{proof}

\begin{theorem}\label{teocacho}Let $\beta>1$ be Pisot and $x\in[0,1]$ be a number such that a $P_\beta$-martingale computable in polynomial time succeeds on the $\beta$-expansion of $x$. Then there is a polynomial time binary martingale that succeeds on the binary expansion of $x$.
\end{theorem}

The proof is the same as that of \cite[Theorem 14]{Fi2013} with \cite[Lemma 15]{Fi2013} replaced with the following:
\begin{lemma}\label{lem:main}
Let $\beta>1$ be Pisot. For any
  polynomial time computable $P_\beta$-martingale $M\colon L(X_\beta)\to\RR_{\geq0}$ with the savings
property there is a classic martingale $N\colon \{0,1\}^*\to\RR_{\geq0}$ such that $N$ is polynomial time computable, and  whenever $M$ succeeds on $s\in X_\beta$, and $x\in[0,1]$ is such that $p_\beta(x)=s$ then $N$ succeeds on the fractional binary expansion of $x$.
\end{lemma}

\begin{proof}
%Let $Y\in \{0,1\}^\NN$ be the binary fractional expansion of $x$.
%
%If $x$ is rational then $Y$ is eventually periodic, and hence
%computable in linear time. In this case we define $N:\{0,1\}^*\to
%\RR_{\geq 0}$ as follows:
%$$
%N(\tau)=
%\begin{cases}
%2^{|\tau|}& \mbox{if $\tau \prec Y$}\\
%0 & \mbox{otherwise},
%\end{cases}
%$$
%which is computable in polynomial time and clearly succeeds on $Y$. In case $x$ is irrational and $\beta$-adic then it is polynomial time computable, and so is $Y$. The same martingale $N$ succeeds on $Y$.

%We assume, then, that $x$ is irrational and not $\beta$-adic.
%

By Proposition~\ref{prop:cdf-polytime}, there is a polynomial time computable function $\widehat \cdfM\colon \Sigma^*_\beta\times\NN\to\QQ$ such that $|\widehat \cdfM(\tau,i)-\cdfM(\rea{\tau})|\leq 2^{-i}$.

%By Proposition~\ref{cor:lipschitz} let $k>0$ be such that for every
%$z,y\in[0,1]$, if $y-z\leq2^{-k}$ then
%\begin{equation}\label{eqn:use-Lipschitz}
%\cdfM(y)-\cdfM(z)\leq -k\cdot(y-z)\cdot\log(y-z).
%\end{equation}

Define the classical martingale
$N\colon \{0,1\}^*\to\RR_{\geq 0}$ as $N(\tau)=(\cdfM(p_2)-\cdfM(p_1))/2^{-|\tau|}$, where $p_1 =\reabin{\tau}$ and $p_2 =\reabin{\tau}+2^{-|\tau|}$. $N$ has a polynomial time computable approximation $\widehat N\colon \{0,1\}^*\times\NN\to\QQ$, defined by
$$
\widehat N(\tau,i) = \frac{\widehat\cdfM(\tau_2,i+2)-\widehat\cdfM(\tau_1,i+2)}{2^{|\tau|}},
$$
where for $j=1,2$ the string $\tau_j\in\Sigma_\beta^*$ is an approximation of $p_j$ with error $2^{-2v-1}$, for $v=i+2+k$. By Proposition~\ref{prop:approx} and the definition of $N$, we conclude that $N$ is polynomial time computable. The proof that $|N(\tau)-\widehat N(\tau,i)|\leq 2^{-i}$ is the same as that of \cite[Fact 16 in the proof of Lemma 15]{Fi2013}, using Proposition~\ref{cor:lipschitz} instead of \cite[Proposition 12]{Fi2013} for the Lipschitz condition.
\end{proof}

We finally arrive to the main theorem of his section:

\begin{theorem}
Let $\beta>1$ be Pisot. If the fractional binary expansion of $x\in[0,1]$ is polynomial time random then it is normal to base $\beta$.
\end{theorem}

\begin{proof}
We proceed by contradiction. Suppose that $Y\in \{0,1\}^\NN$, the fractional binary expansion of $x$, is
not polynomial time random to base~$\beta$. By Lemma~\ref{lem:not-normal-then-p-mart-is-poly}, there is  a polynomial time computable $P_\beta$-martingale $M$ which succeeds on $s=p_\beta(x)$, and  by
Lemma~\ref{lem:convert-to-savings-property} there is a polynomial time computable $P_\beta$-martingale $\widetilde M$ with the savings property that succeeds on all the sequences $M$
succeeds on, in particular on $s$. By Theorem~\ref{teocacho}, $Y$ is not polynomial time random.
\end{proof}

\bibliographystyle{plain}
\bibliography{biblio}
%\nocite{Blan89}
%\nocite{Parry60}
%\nocite{Kitchens98}

\end{document}